\documentclass{amsproc}
\usepackage{graphicx}
\usepackage{subfig}
\usepackage{amsmath}
\usepackage{amsfonts}
\usepackage{amssymb}%
\usepackage{amsthm}
\usepackage{amscd}
\usepackage{url}

\usepackage{graphicx} 

\newtheorem{theorem}{Theorem}
\newtheorem{lemma}[theorem]{Lemma}
\newtheorem{proposition}[theorem]{Proposition}

\newtheorem{corollary}[theorem]{Corollary}

\theoremstyle{remark}

\newcommand{\Q}{\mathbb{Q}}
\newcommand{\R}{\mathbb{R}}

\newcommand{\Z}{\mathbb{Z}}
\newcommand{\PP}{\mathbb{P}}
\newcommand{\T}{\Theta}
\DeclareMathOperator{\val}{val}
\newcommand{\comment}[1]{} 

\title{Elliptic Curves in Honeycomb Form}

\author{Melody Chan}
\address{Department of Mathematics, University of California,
Berkeley, CA 94720, USA}
\email{mtchan@math.berkeley.edu}

\author{Bernd Sturmfels }
\address{Department of Mathematics, University of California,
Berkeley, CA 94720, USA}
\email{bernd@math.berkeley.edu}

\begin{document}
\date{}

\begin{abstract}
A plane cubic curve, defined over a 
field with valuation,  is in honeycomb form if its
tropicalization exhibits the standard hexagonal cycle.
We explicitly compute such representations from a given
$j$-invariant with negative valuation, we give an analytic characterization
of elliptic curves in honeycomb form, and we offer
a detailed analysis of the tropical group law on 
such a curve.
\end{abstract}

\maketitle

\section{Introduction}

Suppose $K$ is a field with a nonarchimedean valuation ${\rm val}:K^*\rightarrow\mathbb{R}$,
 such as the rational numbers
$\Q$ with their $p$-adic valuation for some prime $p \geq 5$ or
 the rational functions $ \Q(t)$ with the $t$-adic valuation.
Throughout this paper, we shall assume that the
residue field of $K$ has characteristic different from $2$ and $3$.

We consider a  ternary cubic polynomial
whose coefficients  $c_{ijk}$ lie in $K$:
\begin{equation}
\label{eq:cubic}
\begin{matrix}
f(x,y,z) & = &
   \,\, c_{300}x^3+c_{210}x^2y +c_{120}xy^2+c_{030}y^3+c_{021}y^2z \\
& &  +\,c_{012}yz^2+c_{003}z^3+c_{102}xz^2+c_{201}x^2z+c_{111}xyz  .
\end{matrix} 
\end{equation}
Provided the discriminant of $f(x,y,z)$ is non-zero, this cubic
represents an elliptic curve $E$ in the projective plane $\PP^2_K$.
The group ${\rm GL}(3,K)$ acts on the projective space
$\mathbb{P}^9_K$ of all cubics.
The field of rational invariants under this action 
is generated by  the familiar {\em j-invariant}, which we can write explicitly 
(with coefficients in $\mathbb{Z}$) as
\begin{equation}
\label{eq:jinvariant}
 j(f) \,\, = \,\,\,
\frac{\hbox{a polynomial  of degree $12$ in the $c_{ijk}$ having $1607$ terms}}{\hbox{a polynomial 
 of degree $12$ in the $c_{ijk}$ having $2040$ terms}}.
\end{equation}

The Weierstrass normal form
of an elliptic curve can be obtained from $f(x,y,z)$ by applying a matrix
in ${\rm GL}(3,\overline{K})$.
From the perspective of  tropical geometry, however,
the Weierstrass form  is too limiting:
its tropicalization never has a cycle.
One would rather have a model for plane cubics
whose tropicalization looks like the graphs
 in Figures~\ref{fig:eins}, \ref{fig:law1} and \ref{fig:line}.
 If this holds then
  we say that $f$ is in {\em honeycomb form}.  Cubic curves in honeycomb form are the central object of interest in this paper.

  Honeycomb curves of arbitrary degree were studied in
   \cite[\S 5]{speyer1}; they are    dual to the standard triangulation of the Newton polygon of $f$. 
   For cubics in honeycomb form,
    by \cite{kmm}, the lattice length of the hexagon 
 equals $- {\rm val}(j(f))$. Moreover, by
 \cite{bpr}, a honeycomb cubic faithfully represents a subgraph of the Berkovich curve $E^{an}$.

A standard Newton subdivision argument \cite{ms} shows that 
a cubic $f$ is in honeycomb form if and only if the following nine 
scalars in $K$ have positive valuation:
\begin{eqnarray}
\label{eq:sixratios}
\frac{c_{021} c_{102}}{c_{111} c_{012}} , \frac{c_{012} c_{120}}{c_{111} c_{021}},
\frac{c_{201} c_{012}}{c_{111} c_{102}}, \frac{c_{210} c_{021}}{c_{111} c_{120}},
\frac{c_{102} c_{210}}{c_{111} c_{201}} , \frac{c_{120} c_{201}}{c_{111} c_{210}},
\end{eqnarray}
\begin{eqnarray}
\label{eq:threeratios}
\frac{c_{111}c_{003}}{c_{012}c_{102}}  \,,\, \frac{c_{111} c_{030}}{c_{021}c_{120}}\,,\,
  \frac{c_{111}c_{300}}{c_{201}c_{210}}. 
\end{eqnarray}
If the six ratios in (\ref{eq:sixratios}) have the same  positive valuation, and also
the three ratios in (\ref{eq:threeratios}) have the same positive valuation,
then we say that $f$ is in {\em symmetric honeycomb form}.  So $f$ is in symmetric honeycomb form if and only if the lattice lengths of the six sides of the hexagon are equal, and the lattice lengths of the three bounded segments coming off the hexagon are also equal, as in Figure~\ref{fig:eins} on the right.

\begin{figure}
\centering
\subfloat{\includegraphics[width=0.5\textwidth]{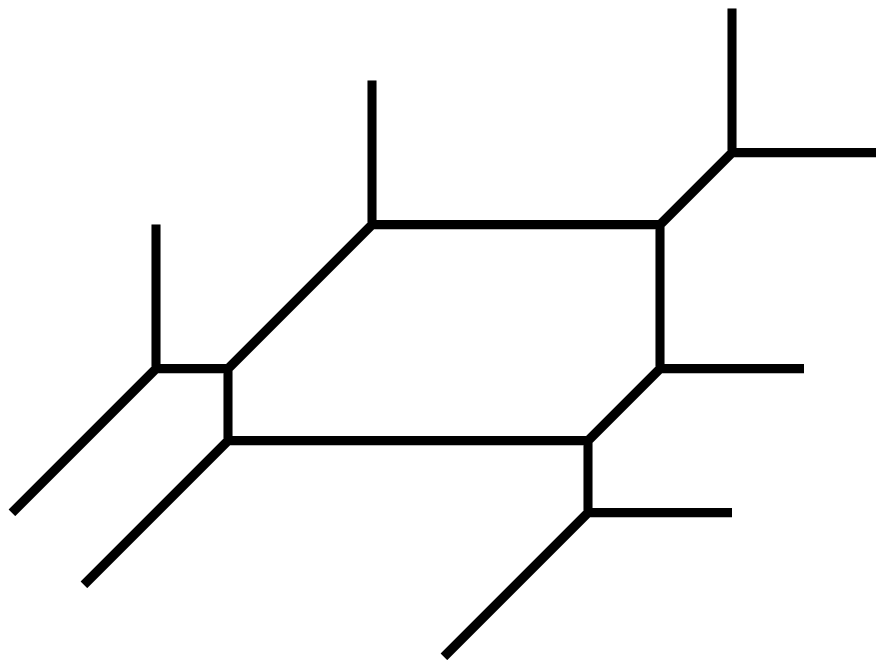}} 
\qquad
\subfloat{\includegraphics[width=0.4\textwidth]{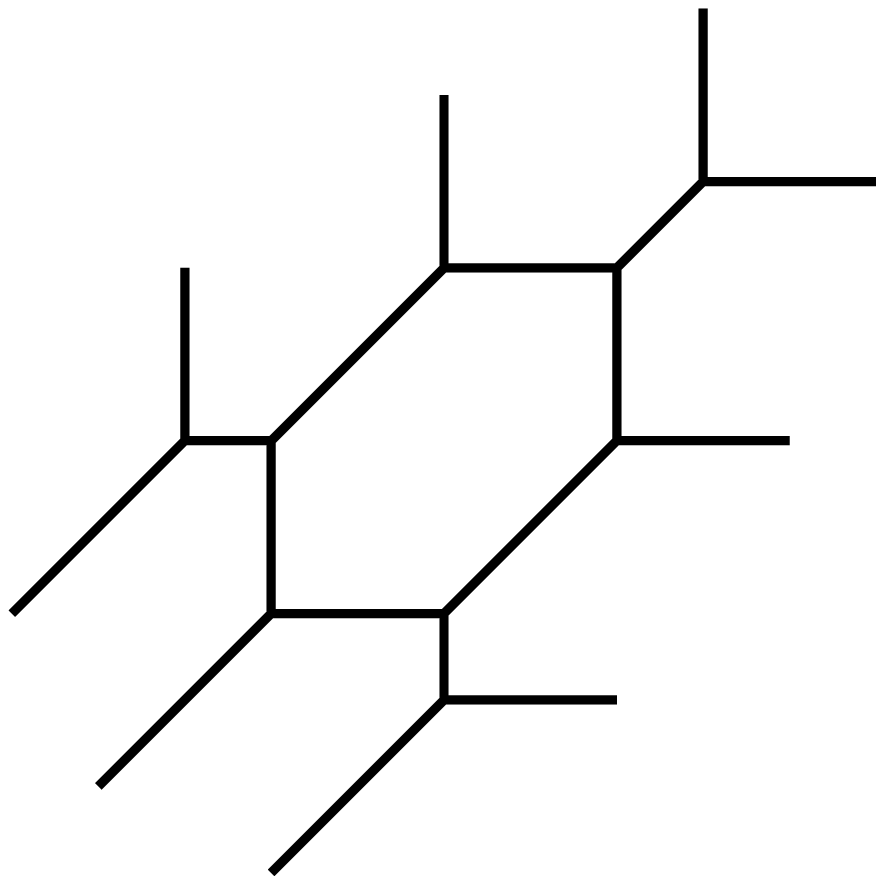}} 
\vskip -0.1in
\caption{Tropicalizations of plane cubic curves
in honeycomb form. The curve on the right is in
symmetric honeycomb form; the one on the left is not symmetric.}
\label{fig:eins}
\end{figure}

Our contributions in this paper are as follows.
In Section 2 we focus on symmetric honeycomb cubics.
We present a symbolic algorithm whose input is an arbitrary cubic
$f$ with ${\rm val}(j(f)) < 0$ and
whose output is a $3 \times 3$-matrix $M$
such that $f \circ M$ is in symmetric honeycomb form.
This answers a question raised by
Buchholz and Markwig (cf.~\cite[\S 6]{buch}).
We pay close attention to the arithmetic of the entries of $M$.
Our key tool is the relationship between honeycombs and
 the {\em Hesse pencil}~\cite{AD, nobe}.
Results similar to those in Section 2 were obtained independently by Helminck \cite{helminck}.

Section 3 discusses the {\em Tate parametrization} \cite{silverman}
of elliptic curves using theta functions. Our approach is similar to
that used by Speyer in \cite{speyer2} for lifting tropical curves.
 We present an analytic characterization
of honeycomb cubics with prescribed $j$-invariant, and we give a numerical algorithm
for computing such cubics.

Section 4 explains a combinatorial rule for the {\em tropical group law}
on a honeycomb cubic $C$.
Our object of study is the tropicalization of the 
surface $\,\{u,v,w \in C^3 \, |\, u \star v \star  w =\rm{id}\} \,\subset \,(\mathbb{P}^2)^3$.
Here  $\star$ denotes multiplication on $C$.
  We explain how to compute this tropical surface  in $(\mathbb{R}^2)^3$.
  See Corollary \ref{cor:eleven} for a concrete instance.
Our results
  complete  the partially defined group law found by Vigeland~\cite{vige}.

Practitioners of computational algebraic geometry are well
aware of the challenges involved in working with algebraic varieties
over a valued field~$K$. One aim of this article is to demonstrate
how these challenges can be overcome in practice, at least for the basic case of
elliptic curves. In that sense, our paper can be read as a
computational algebra supplement  to the work
of Baker, Payne and Rabinoff~\cite{bpr}.

Many of our methods
have been implemented in \textsc{Mathematica}.
Our code and the examples in this paper
can be   found at our
supplementary materials website
$$ \hbox{{\tt www.math.berkeley.edu/$\sim$mtchan/honeycomb.html}} $$
In our test implementations, the input data are assumed to lie in the field 
$K = \mathbb{Q} (t) $, and scalars in $\overline{K}$ are
represented  as truncated Laurent series 
with coefficients in $\overline{\mathbb{Q}}$.
This  is analogous to the representation of
scalars in $\mathbb{R}$ by floating point  numbers.

\section{Symmetric Cubics}

We begin by establishing the existence of symmetric honeycomb forms
for elliptic curves whose $j$-invariant has negative valuation. Consider a symmetric  cubic 
\begin{equation}
\label{eq:specialf1}
g \,\,\, = \,\,\,
a \cdot (x^3+y^3+z^3) \, + \,
b \cdot (x^2 y + x^2 z + x y^2 + x z^2 + y^2 z + y z^2) \,
+  xyz . 
\end{equation}
The conditions in (\ref{eq:sixratios})-(\ref{eq:threeratios}) imply that
$g$ is in symmetric honeycomb form if and only~if 
\begin{equation}
\label{eq:specialf2}
 {\rm val}(a) \, > \, 2 \cdot {\rm val}(b) \,  > \, 0.
  \end{equation}
Our aim in this section is to transform arbitrary cubics (\ref{eq:cubic}) 
to symmetric cubics in honeycomb form. In other words, we seek
to achieve   both (\ref{eq:specialf1}) and (\ref{eq:specialf2}).
Note that $a=0$ is allowed by the valuation inequalities (\ref{eq:specialf2}),
but $b$ must be non-zero in (\ref{eq:specialf1}).
The classical {\em Hesse normal form} of \cite[Lemma 1]{AD},
whose tropicalization was examined recently by Nobe \cite{nobe},
is therefore ruled out by the honeycomb~condition.

\begin{proposition} \label{thm1}
Given any two scalars $ \iota$ and $ a$ in $K$ with ${\rm val}(\iota) < 0$
and ${\rm val}(a) + {\rm val}(\iota) > 0$,
there exist precisely six elements $b $ in the algebraic closure $\overline{K}$,
defined by an equation of degree $12$ over $K$, such that the cubic $g$ 
above has $j$-invariant $\,j(g) = \iota \,$ and is in symmetric honeycomb form.
\end{proposition}

\begin{proof}
First, consider the case $a = 0$, so that ${\rm val}(a) = \infty$.
By specializing (\ref{eq:jinvariant}), we deduce that the
 $j$-invariant of $g  =  b(x^2 y + x^2 z + x y^2 + x z^2 + y^2 z + y z^2 ) +   xyz $~is
\begin{equation}
\label{eq:mustsolve1}
j(g) \,\,\, = \,\,\,  \frac{
(48 b^3 - 24 b^2 + 1)^3}
{ b^6  (2b-1)^3 (3b-1)^2  (6b+1)
}.
 \end{equation}
 Our task is to find $b \in \overline{K}$ such that
   $j(g) = \iota$. The expansion of this equation  equals
  \begin{equation}
  \label{eq:ceq}
  \begin{matrix}
432 \iota b^{12} - 864 \iota b^{11} + 648 \iota b^{10}
-(208 \iota + 110592)b^9
+(15 \iota+165888) b^8 & \\ \!\! + \,
(6 \iota -82944) b^7 
  -( \iota-6912)b^6
  +  6912 b^5-1728 b^4-144 b^3+72b^2-1& \!\! = \,\, 0.
  \end{matrix}
    \end{equation}
We examine  the  Newton polygon of  this equation.
    It is independent of $K$ because the characteristic of the residue field of $K$ is not $ 2$ or $3$.
     Since $\iota$ has negative valuation, we see that (\ref{eq:ceq}) has
six solutions $b \in \overline{K}$ with ${\rm val}(b) = 0$ and six solutions $b$
with ${\rm val}(b) = -{\rm val}(\iota)/6$. 
The latter six solutions
are indexed by the choice of a sixth root  of $\iota^{-1}$. They
  share the following  expansion as a Laurent series~in~$\iota^{-1/6}$:
    $$
 \begin{matrix} b &  =  &  
 \iota^{-1/6}
+ \iota^{-1/3}
-5 \iota^{-1/2}
-7 \iota^{-2/3}
+30 \iota^{-5/6}
+43 \iota^{-1}
-60 \iota^{-7/6}
-15 \iota^{-4/3} \\ & & 
-731 \iota^{-3/2}
-1858 \iota^{-5/3}
+11676 \iota^{-11/6}
+22091 \iota^{-2}
-30612 \iota^{-13/6} + \cdots
\end{matrix}
$$
These six values of $b$ establish the assertion in Proposition \ref{thm1} when $a = 0$. 

Now suppose $a \not= 0$. Then
our equation $j(g) = \iota$ has the  more complicated~form
\begin{equation}
\label{eq:mustsolve2} 
\frac{(6a-1)^3 (72 a b^2-48 b^3-36 a^2 + 24 b^2 - 6a -1)^3}{
(3 a+6 b+1)(3 a-3b+1)^2 (9a^3-3ab^2+2b^3-3a^2-b^2+a)^3} \,\,\, = \,\,\, \iota.
\end{equation}
Our hypotheses on $K$ and $a$ ensure that ${\rm val}(a)$ is large enough so as
to not interfere with the lowest order terms when solving this equation for $b$.
In particular, the degree $12$ equation in the unknown $b$ with coefficients in $K$
resulting from (\ref{eq:mustsolve2})
has the same Newton polygon as  equation (\ref{eq:ceq}).
As before,  this equation has $12$ solutions $b = b(\iota,a)$ 
that are scalars in $\overline{K}$, and six of the solutions satisfy  ${\rm val}(b) = 0$
while the other six satisfy ${\rm val}(b) = -{\rm val}(\iota)/6$.
The latter six establish our assertion.
\end{proof}

We have proved the existence of a symmetric honeycomb form
for any nonsingular cubic whose $j$-invariant has negative valuation. Our main
goal in what follows is to describe an algorithm for computing a
$3 \times 3$ matrix that transforms a given cubic into that form.
Our method is to compute the nine inflection points of each cubic  and find a suitable
projective transformation that takes one set of points to the other.
Computing the inflection points is a relatively easy task in the special
case of symmetric cubics.
The result of that computation is the following lemma.

\begin{lemma}
Let $C$ be a nonsingular cubic curve defined over $K$ by a symmetric polynomial $g$ as in (\ref{eq:specialf1}), 
fix a primitive third root of unity $\xi$ in $\overline{K}$, and set
\begin{equation}
\label{eq:uratio}
 \omega \quad = \quad \frac{\phantom{-}3 a+6 b+1}{-3 a+3b-1}.  
 \end{equation}
Then the nine inflection points of $C$ in $\PP^2$ are given by the rows of the matrix
\begin{equation}
\label{eq:Aomega}
A_\omega \quad = \quad \begin{bmatrix}
1 & -1 & 0 \\
 1 & 0 & -1 \\
  0 & 1 & -1 \\
1 + \omega^{1/3} & 1 + \xi \omega^{1/3} & 1 + \xi^2 \omega^{1/3} \\
1 + \xi \omega^{1/3} & 1 + \xi^2 \omega^{1/3} & 1 + \omega^{1/3} \\
1 + \xi^2 \omega^{1/3} & 1+ \omega^{1/3} & 1 + \xi \omega^{1/3} \\
1 + \xi \omega^{1/3} & 1+ \omega^{1/3} & 1 + \xi^2 \omega^{1/3} \\
1 + \omega^{1/3} & 1 + \xi^2 \omega^{1/3} & 1 + \xi \omega^{1/3} \\
1 + \xi^2 \omega^{1/3} & 1+\xi \omega^{1/3} & 1 + \omega^{1/3} \\
\end{bmatrix}.
\end{equation}
\end{lemma}

The matrix $A_\omega$ has precisely the following vanishing $3\times 3$-minors:
\begin{equation}
\label{eq:hesseconf}
 123 , \, 147 , \, 159 , \, 168 , \, 249 , \, 258 , \, 267 , \, 348 , \, 357 , \, 369 , \, 456, \, 789. 
 \end{equation}
 This list of triples is the classical {\em Hesse configuration} of  $9$ points and $12$ lines.
 
Next, for an arbitrary nonsingular cubic $f$ as in (\ref{eq:cubic}), the
nine inflection points can be expressed in radicals
in the ten coefficients $c_{300}, c_{210}, \ldots, c_{111}$,
since their Galois group is solvable \cite[\S 4]{AD}.
How can we compute these inflection points?  Consider the {\it Hesse pencil} ${\rm HP}(f)
=\{s\cdot f+s'\cdot H_f:s,s'\in K\}$ of plane cubics spanned by $f$ and its Hessian $H_f$.  Each cubic in 
${\rm HP}(f)$ passes through the nine inflection points of $f$ since both $f$ and $H_f$ do, and in fact every such cubic is in ${\rm HP}(f)$.  In particular, the four systems of three lines through the nine points are precisely the four reducible members of ${\rm HP}(f)$.  Indeed, if $g=l\cdot h\in {\rm HP}(f)$ where $l$ is a line passing through three inflection points, then $h$ passes through the remaining six and thus must itself be two lines by B\'ezout's Theorem.  So we may compute any two of the four such systems of three lines, and take pairwise intersections of their lines to obtain the nine desired inflection points.  This algorithm was extracted from Salmon's book
\cite{salmon}, and it runs in exact
arithmetic. We now make it more precise.

We introduce four unknowns $u,v,w,s$, and we consider 
the condition that a cubic $s \cdot f + H_f$ is divisible by the linear form $u x + v y + w z$.
 That condition translates into a 
system of polynomials that are cubic in the unknowns $u,v,w$ and linear in $s$.
We derive this system by specializing the following universal solution,
 found by  a {\tt Macaulay2} computation which is posted on our
supplementary materials website.

\begin{lemma}
The condition that a linear form $u x + v y + w z$ divides a
cubic (\ref{eq:cubic}) is given by a prime ideal in the polynomial ring
$K[u,v,w,c_{300}, c_{210},\ldots,c_{003}]$
in $13$ unknowns. This prime ideal is of codimension $4$ and degree $28$. It has
 $96$ minimal generators, namely $25$ quartics, $15$ quintics, $21$ sextics and $35$ octics.
\end{lemma}

Consider the polynomials in $u,v,w,s$ that are obtained by specializing the
$c_{ijk}$ in the
$96$ ideal generators above to the coefficients of $s \cdot f + H_f$.
After permuting coordinates if necessary, we may set $w=1$
and work with the resulting polynomials in $u,v,s$.
The lexicographic Gr\"obner basis of their ideal 
 has the  special form
$$ 
\bigl\{\,
\underline{s^4} + \alpha_2 s^2 + \alpha_1 s + \alpha_0 \, ,  \,\,
\underline{v^3} + \beta_2 (s) v^2 + \beta_1(s) v + \beta_0 (s) \, , \,\,
\underline{u} + \gamma(s,v) \,
\bigr\},
$$
where the $\alpha_i$ are constants in $K$,
the $\beta_j$ are univariate polynomials,
and $\gamma$ is a bivariate polynomial.
These equations have $12$ solutions 
$ (s_i,u_{ij},v_{ij}) \in (\overline{K})^3$ where $i=1,2,3,4$ and $j=1,2,3$.
The  leading terms in the Gr\"obner basis reveal that 
the coordinates of these solutions can be expressed in radicals over $K$, since we need only solve a quartic in $s$, a cubic in $v$, and a degree 1 equation in $u$, in that order.

For each of the nine choices of $j,k \in \{1,2,3\}$, the two linear equations
$$ u_{1j} x + v_{1j} y + z \, = \, u_{2k} x + v_{2k} y + z \,\, = \,\,  0 $$
have a unique solution $(b^{jk}_1: b^{jk}_2: b^{jk}_3)$
in the projective plane $\PP^2 $ over $\overline{K}$.
We can write its coordinates $b^{jk}_l$ in radicals over $K$.
Let $B$ denote the $9 \times 3$-matrix 
whose rows are the vectors $(b^{jk}_1, b^{jk}_2, b^{jk}_3)$ for $j,k \in \{ 1,2,3\}$.
While the entries of  $B$ have been
  written in radicals over $K$, they  can also be represented as 
 formal series in the completion of 
$\overline{K}$, which we can approximate by a suitable truncation.

To summarize our discussion up to this point:
we have shown how to compute the inflection points of a plane cubic,
and we have written them as the rows of a $9 \times 3$-matrix $B$
whose entries are expressed in radicals over $K$.
For the special case of symmetric cubics, the 
specific $9 \times 3$-matrix $A_\omega$ in (\ref{eq:Aomega})
gives the inflection points.

\smallskip

Now, we return to our main goal. Suppose we are given a 
nonsingular ternary cubic $f$ whose $j$-invariant $\iota = j(f)$ has negative valuation.
We then choose  $a,b \in \overline{K}$ as prescribed in Proposition \ref{thm1}, and we define $\omega$ by the ratio in (\ref{eq:uratio}).  The scalars $a$ and $b$ define a symmetric honeycomb cubic $g$ as in (\ref{eq:specialf1}).
Let $\mathcal{A}_\omega$ and $\mathcal{B}$ denote the sets of
inflection points of the cubic curves $V(g)$ and $V(f)$ respectively.
Thus  $\mathcal{A}_\omega$ and $\mathcal{B}$ 
are unordered $9$-element subsets of $\PP^2$,
represented by the rows of
our matrices $A_\omega$ and $B$. There exists an automorphism $\phi$ of $\mathbb{P}^2$ taking $V(f)$ to $V(g)$, since their $j$-invariants agree.  Clearly, any such automorphism $\phi$  takes 
$\mathcal{B}$ to $\mathcal{A}_\omega$.

We write $\mathcal{B} = \{b_1,b_2,\ldots,b_9\}$, where
the labeling is such that $b_i,b_j,b_k$ are collinear in $\PP^2$
if and only if $ijk$ appears on the list (\ref{eq:hesseconf}).
The automorphism group of the
Hesse configuration   (\ref{eq:hesseconf})  has order  $9\cdot8\cdot 6=432$.
 Hence precisely $432$ of
 the $9!$ possible bijections $\mathcal{B} \rightarrow \mathcal{A}_\omega$
respect the collinearities  of the inflection points.
For each such bijection  $\pi_i:  \mathcal{B} \rightarrow \mathcal{A}_\omega$,
$i=1,2,\ldots,432$, we associate a unique projective transformation
$\sigma_i : \PP^2 \rightarrow \PP^2$ by requiring that
$\sigma_i(b_1) = \pi_i(b_1)$,
$\sigma_i(b_2) = \pi_i(b_2)$,
$\sigma_i(b_4) = \pi_i(b_4)$ and
$\sigma_i(b_5) = \pi_i(b_5)$.  We emphasize that $\sigma_i$ may or may not induce a bijection $\mathcal{B} \rightarrow \mathcal{A}_\omega$ on all nine points.
We write $M_i$
for the unique (up to scaling) $3 \times 3$-matrix with
entries in $\overline{K}$ that represents 
the projective transformation $\sigma_i$. 

The simplest version of our algorithm constructs all  matrices
$M_1,M_2,\ldots,M_{432}$. 
One of these matrices, say $M_j$, represents the
 automorphism $\phi$ of $\mathbb{P}^2$
in the second-to-last paragraph.
 The ternary cubics $f\circ M_j$ and $ g$  are equal
up to a scalar.
To find such an index $j$, we simply check,   for each $j \in \{1,2,\ldots,432\}$,
    whether $f\circ M_j$ is in symmetric honeycomb form. 
    The answer will be affirmative for at least one index $j$,
    and we set $M = M_j$.
    This  resolves the question raised 
by Markwig and Buchholz  \cite[\S 6]{buch}.
The following theorem summarizes the problem and our solution.

\begin{theorem} \label{thm:arnehanna}
Let $f$ be a nonsingular cubic with ${\rm val}(j(f)) < 0 $. If  $M$ is
the $3 {\times} 3$-matrix over $\overline{K}$ constructed above then
 $f \circ M$ is a symmetric honeycomb cubic.
\end{theorem}

Next, we discuss a refinement of the algorithm above that reduces the number of matrices to check 
from $432$ to $12$.  It takes advantage of the detailed description of the 
Hessian group in \cite{AD}.
Given a plane cubic $f$, the {\it Hessian group} $G_{216}$ consists of those linear automorphisms of $\mathbb{P}^2$ that preserve the pencil ${\rm HP}(f)$.  This group was first described by C.~Jordan in \cite{jordan}.  The elements of $G_{216}$ naturally act on the subset
$\mathcal{A}_\omega$ of $\mathbb{P}^2$ given by the rows
$a_1,a_2,\ldots,a_9$ of $A_\omega$.
  Of the $432$ automorphisms of (\ref{eq:hesseconf}), precisely half are realized by the action of $G_{216}$.  
    The group $G_{216}$ is isomorphic to the semidirect product $(\mathbb{Z}/3\mathbb{Z})^2
     \rtimes {\rm SL}(2,3)$.  The first factor sends $f$ to itself and permutes 
    $\mathcal{A}_\omega$ transitively.  The second factor, of order $24$, 
 sends $f$ to each of the $12$ cubics in ${\rm HP}(f)$ isomorphic to it.  The quotient of ${\rm SL}(2,3)$ by the 2-element stabilizer of $f$ is isomorphic to 
${\rm PSL}(2,3)$, with $12$ elements. 
Identifying $G_{216}$ with the subgroup of $S_9$ permuting $\mathcal{A}_\omega$,
a set of coset representatives for ${\rm PSL}(2,3)$ inside $G_{216}$ 
consists of the following $12$ permutations (in cycle notation):
\begin{equation}
\label{eq:cosets}
\begin{matrix} {\rm id},\,\, (456)(987), \,\, (654)(789),\,\,
 (2437)(5698), \,\, (246378)(59),\,\,
(254397)(68)\\
(249)(375), 
(258)(963), (2539)(4876),
(852)(369), (287364)(59), (2836)(4975)
\end{matrix}
\end{equation}
With this notation, an example of an automorphism of the Hesse configuration  (\ref{eq:hesseconf}) 
that is  not realized by the Hessian group $G_{216}$ is the permutation $\tau=(47)(58)(69)$.

Here is now our refined algorithm for the last step towards 
Theorem \ref{thm:arnehanna}. Let $f,g,\mathcal{B},
 \mathcal{A}_\omega$ be given
  as above. For any automorphism $\rho$ of (\ref{eq:hesseconf}),
we denote by $\phi_\rho:\mathbb{P}^2\rightarrow\mathbb{P}^2$ the projective transformation 
$\,(b_1, b_2, b_4, b_5) \,\mapsto \, (a_{\rho(1)},a_{\rho(2)},a_{\rho(4)},a_{\rho(5)})$.  
To find a transformation from $f$ to $g$, we proceed as~follows.  
First, we check to see whether $\phi_{\rm id}$ maps 
$\mathcal{B}$ to $\mathcal{A}_\omega$.  (It certainly maps four elements of 
$\mathcal{B}$ to $\mathcal{A}_\omega$ but maybe not all nine).  
If $\phi_{\rm id}(\mathcal{B}) = \mathcal{A}_\omega$ then $\phi_{\rm id}(f)$ is in the
Hesse pencil ${\rm HP}(g)$. This implies that one of the
 $12$ maps $\phi_{\sigma}$, where $\sigma$ runs over (\ref{eq:cosets}),
  takes $f$ to $g$. 
    If $\phi_{\rm id}(\mathcal{B}) \not= \mathcal{A}_\omega$
   then $\phi_{\tau}$ must map $\mathcal{B}$ to 
   $\mathcal{A}_\omega$, since $G_{216}$ has index $2$ in 
   the automorphism group of the Hesse configuration
      and $\tau$ represents the nonidentity coset.  Then one of the $12$ maps $\phi_{\tau\sigma}$,
      where $\sigma$ runs over (\ref{eq:cosets}),  takes $f$ to $g$. 
In either case, after computing $\phi_{\rm id}$, we only have to check $12$ maps, and one of them will work.
 
 We close with two remarks. First, the set of matrices 
$M\in {\rm GL}(3,\overline{K})$ that send a given cubic $f$ into 
honeycomb form is a rigid analytic variety, since the conditions on the entries of $M$ are inequalities in valuations of polynomial expressions therein.  It would be interesting to study this space further.

The second remark concerns
  the arithmetic nature of the output of our algorithm.
The entries of  the matrix $M$  were constructed to
be expressible in radicals over $K(\omega)$, with $\omega$ as in
(\ref{eq:uratio}).
However, as it stands, we do not know whether they can be expressed in radicals
over the ground field $K$. The problem lies in the application of Proposition~\ref{thm1}. Our first step was to chose a scalar $a \in K$ whose valuation is large enough.
Thereafter, we computed $b$ by solving a univariate equation 
of degree $12$. This equation is generally irreducible
with non-solvable Galois group. Perhaps it is possible to
choose $a$ and $b$ simultaneously, in radicals over $K$,
so that $(a,b)$ lies on the curve (\ref{eq:mustsolve2}), but at present,
we do not know how to make~this~choice.

\section{Parametrization and Implicitization}

A standard task of computer algebra is to go back and forth
between parametric and implicit representations of algebraic varieties.
Of course, these transformations  are most transparent
when the variety is rational. If the variety is not
unirational then parametric representations typically involve transcendental functions.
In this section, we use nonarchimedean  theta functions to parametrize
planar cubics, we demonstrate how to implicitize this parametrization,
and we derive an intrinsic characterization of honeycomb cubics in terms of
their nonarchimedean geometry.

In this section we assume that $K$ is an algebraically closed field which is complete with respect to a nonarchimedean valuation.  Fix a scalar $\iota \in K$ with $\val(\iota)<0$.  
According to Tate's classical
 theory \cite{silverman},  
 the unique elliptic curve $E$ over $K$ with $j(E)=\iota$
is analytically isomorphic to $\,K^*/{q^{\mathbb{Z}}}$, where $q \in K^*$ is a particular scalar
with ${\rm val}(q) > 0$, called the {\em Tate parameter} of $E$. The symbol
 $q^{\mathbb{Z}}$ denotes the multiplicative group generated by $q$.  The Tate parameter of $E$ is determined 
 from the $j$-invariant by inverting the power series relation
\begin{equation}
\label{eq:qseries}
j\,\,=\,\,{1\over q} +744 +196884q + 21493760 q^2 + 
 \cdots 
\end{equation}
This relation can be derived and computed to arbitrary precision from the identity
$$j = \frac{(1-48a_4(q))^3} {\Delta(q)}, $$
where the invariant
$a_4$ and the discriminant $\Delta$ are given by
$$a_4(q) \,\, = \,\, -5 \sum_{n \geq 1} \frac{n^3q^n}{1-q^n} 
\qquad \textrm{and} \qquad \Delta (q) \,\,= \,\,q \prod_{n\geq 1} (1-q^n)^{24}.$$
We refer to Silverman's text book \cite{silverman} for an introduction 
to the relevant theory of elliptic curves,
and  specifically to  \cite[Theorems V.1.1, V.3.1]{silverman} for the above results.

Our aim in this section is to work directly with the analytic representation
 $$ E \,\, = \,\,  K^*/{q^{\mathbb{Z}}}, $$
 and to construct its honeycomb embeddings into the plane $\PP^2_K$.
 In our explicit computations, scalars in $K$ are presented as truncated power series
 in a uniformizing parameter. The arithmetic is numerical rather than symbolic.
Thus, this section connects
 the emerging fields of tropical geometry and
  numerical algebraic geometry.

 By a {\it holomorphic function} on $K^*$ we mean a formal Laurent series $\sum a_n x^n$ which converges for every $x\in K^*$.  A {\it meromorphic function} is a ratio of two holomorphic functions; they have a well-defined notion of zeroes and poles as usual.  A {\it theta function} on $K^*$, relative to the subgroup $q^\mathbb{Z}$, is a meromorphic function on $K^*$ whose divisor is periodic with respect to $q^\mathbb{Z}$.  Hence theta functions on $K^*$ represent divisors on $E$.  The  {\em fundamental theta function} $\Theta: K^* \rightarrow K$ is 
 defined by
$$\Theta(x) \,\,= \,\, \prod_{n>0}(1-q^nx)\prod_{n\geq 0} (1-{q^n\over x}).$$ 
 Note that $\Theta$ has a simple zero at the identity of $E$ and no other zeroes or poles.
Furthermore, given  any $a \in K^*$, we define the shifted theta function 
$$\Theta_a(x) \,\, = \, \, \Theta(x/a).$$ 
The function 
$\Theta_a$ represents the divisor $[a]\in {\rm Div} E$, where $[a]$ denotes the point of
the elliptic curve $E$ represented by $a$.  One can also check that $\Theta_a(x/q) = -{x\over a} \Theta_a(x).$

Now suppose $D= n_1p_1 + \dots +n_sp_s$ is a divisor on $E $
 that satisfies $\deg(D)=0$ and $p_1^{n_1}\cdots p_s^{n_s}=1$, as an equation in the multiplicative group $K^*/q^{\mathbb{Z}}$.  We can use theta functions to exhibit $D$ as a principal divisor, as follows.  Pick lifts $\tilde{p}_1, \dots, \tilde{p}_s \in K^*$ of $p_1, \ldots, p_s$, respectively, such that $\tilde{p}_1^{n_1}\cdots \tilde{p}_s^{n_s} =1$ as an equation in $K^*$. Let
$$ f(x)  \,\,= \,\, \Theta_{  \tilde{p}_1}(x)^{n_1} \cdots  \Theta_{  \tilde{p}_s}(x)^{n_s}.$$
This defines a function $f : K^* \rightarrow K$ that is $q$-periodic because
$$ f(x/q) \,\,= \,\,\frac{(-x)^{n_1 + \dots +n_s}}{ \tilde{p}_1^{n_1}\cdots \tilde{p}_s^{n_s}} f(x) 
\,\ = \,\, (-x)^{n_1 + \dots +n_s} f(x)
\,\, = \,\, f(x). $$
The last equation holds because we assumed that
${\rm deg}(D) = n_1 + \cdots + n_s$ is zero.
We conclude that $f$
descends to a meromorphic function on $K^*/q^{\mathbb{Z}}$ with divisor $D$.

We now present a parametric representation of plane cubic curves
that will work well for honeycombs.
In what follows, we write $(z_0:z_1:z_2)$ for the coordinates on
$\mathbb{P}^2$.
Fix scalars $a,b,c,p_1,\ldots,p_9 $ in $K^*$ that  satisfy the conditions
\begin{equation}
\label{eq:transpara}
p_1p_2p_3 \,\,=\,\,p_4p_5p_6 \,\, = \,\, p_7p_8p_9
\quad \hbox{and} \quad
p_i/p_j\not\in q^\Z \,\textrm{ for }\,i\ne j.
\end{equation}
The following defines a map from $E = K^*/q^{\mathbb{Z}}$ into
the projective plane $\mathbb{P}^2$ as follows:
\begin{equation}
\label{eq:param}
x \,\,\,\,\mapsto \,\,\,
\bigl( a\cdot \T_{p_1}\T_{p_2}\T_{p_3}(x) 
:b\cdot \T_{p_4}\T_{p_5}\T_{p_6}(x)
:c\cdot \T_{p_7}\T_{p_8}\T_{p_9}(x)  \bigr).
\end{equation}
This map embeds the elliptic curve 
$E = K^*/q^{\mathbb{Z}}$ analytically
 as a plane cubic:

\begin{lemma}
\label{thm:analytic1}
If the image of the map (\ref{eq:param}) has three distinct intersection points with each
of the three coordinate lines $\{z_i = 0\}$, then it is a cubic curve in $\mathbb{P}^2$.
Every nonsingular  cubic 
with this property and having Tate parameter $q$
  arises this~way.
\end{lemma}

\begin{proof}
By construction, the following two functions $K^* \rightarrow K$ are $q$-periodic:
\begin{equation}
\label{eq:f_and_g}
f(x) \,\, = \,\,
\frac{a\cdot \T_{p_1}\T_{p_2}\T_{p_3}(x)}{c\cdot \T_{p_7}\T_{p_8}\T_{p_9}(x)}
\quad \hbox{and} \quad
g(x) \,\, = \,\,
\frac{b\cdot \T_{p_4}\T_{p_5}\T_{p_6}(x)}{c\cdot \T_{p_7}\T_{p_8}\T_{p_9}(x)}.
\end{equation}
Hence $f$ and $g$ descend to meromorphic functions on the
elliptic curve $E= K^*/q^{\mathbb{Z}}$. The map (\ref{eq:param})
can be written as $\,x \mapsto (f(x): g(x) : 1) \,$ and this
defines a map from $E$ into $\PP^2$. The divisor
$D = p_7 + p_8 + p_9$ on $E$ has degree $3$.
By the Riemann-Roch argument in \cite[Example 3.3.3]{hartshorne},
its space of sections $L(D)$ is $3$-dimensional.
Moreover, the assumption about having three distinct intersection points
 implies that the meromorphic functions $f,g$ and $1$ form a
basis of the vector space $L(D)$.
The image of $E$ in
$\PP(L(D)) \simeq \PP^2$ is a cubic curve because
$L(3D)$ is $9$-dimensional. 

For the second statement, we take $C$ to be any nonsingular cubic curve
with Tate parameter $q$ that has distinct intersection points with the three coordinate lines
in $\mathbb{P}^2$.
There exists a morphism $\phi$ from the abstract elliptic curve 
$E= K^*/q^{\mathbb{Z}}$ into $\PP^2$ whose image equals $C$.
Let $\{p_1,p_2,p_3\}$, $\{p_4,p_5,p_6\}$
and $\{p_7,p_8,p_9\}$ be the
preimages under $\phi$ of the triples 
$C \cap \{z_0=0\}$, $C \cap \{z_1=0\}$
and $C \cap \{z_2=0\}$ respectively.
The divisors
$D_1 = p_1 + p_2 + p_3 - p_7-p_8-p_9$
and $D_2 = p_4+ p_5+p_6 - p_7-p_8-p_9$
are principal, and hence
$\, \frac{p_1 p_2 p_3}{p_7p_8p_9}  =
\frac{p_4 p_5 p_6}{p_7p_8p_9}  = 1 \,$
in the multiplicative group
$E= K^*/q^{\mathbb{Z}}$.
We choose preimages 
$\tilde p_1, \tilde p_2,\ldots,\tilde p_9 $ in $K^*$ such that
(\ref{eq:transpara}) holds for these scalars.

Our map $\phi$ can be written in the form
$x \mapsto (f(x):g(x):1)$ where ${\rm div}(f) = D_1$ and ${\rm div}(g) = D_2$.
By the Abel-Jacobi Theorem (cf.~\cite[Proposition 1]{roquette}),
the function $f$  is uniquely determined, up to a multiplicative
scalar, by the property ${\rm div}(f) = D_1$, and similarly for $g$ and $D_2$.
Then there exist $\gamma_1,\gamma_2 \in K^*$ such that
$$
f(x) \,\, = \,\,
\gamma_1 \frac{\T_{p_1}\T_{p_2}\T_{p_3}(x)}{ \T_{p_7}\T_{p_8}\T_{p_9}(x)}
\quad \hbox{and} \quad
g(x) \,\, = \,\, 
\gamma_2
\frac{\T_{p_4}\T_{p_5}\T_{p_6}(x)}{\T_{p_7}\T_{p_8}\T_{p_9}(x)}.
$$
We conclude that $\phi$ has the form
(\ref{eq:param}).
\end{proof}

It is a natural ask to what extent the
parameters in the representation (\ref{eq:param})  of a plane cubic are unique. 
The following result answers this question.

\begin{proposition} \label{prop:whensame}
Two vectors $(a,b,c,p_1,\ldots,p_9)$ and $(a',b',c',p'_1,\ldots,p'_9)$ 
in $(K^*)^{12}$, both satisfying
(\ref{eq:transpara}), define the same plane cubic if and only if
the latter vector can be obtained from the former by combining the following operations:
\begin{enumerate}
\item[(a)] Permute the sets $\{p_1,p_2,p_3\}$, $\{p_4,p_5,p_6\}$ and $\{p_7,p_8,p_9\}$.
\item[(b)] Scale each of $a,b$ and $c$ by the same multiplier $\lambda \in K^*$.
\item[(c)] Scale each $p_i$ by the same multiplier $\lambda \in K^*$.
\item[(d)] Replace each $p_i$ by its multiplicative inverse $1/p_i$.
\item[(e)] Multiply each $p_i$ by $q^{n_i}$ for some $n_i \in \mathbb{Z}$,
where $n_1{+}n_2{+}n_3=
n_4{+}n_5{+}n_6=n_7{+}n_8{+}n_9$,
and  set $a' = p_1^{n_1}p_2^{n_2}p_3^{n_3}a$,
$b' = p_4^{n_4}p_5^{n_5}p_6^{n_6}b$,
$c' = p_7^{n_7}p_8^{n_8}p_9^{n_9}c$.
\end{enumerate}
\end{proposition}

\begin{proof}
Clearly the relabeling in (a) and the scaling in (b) preserve the curve $C \subset \PP^2$. For (c), we note
 that scaling each $p_i$ by the same constant $\lambda\in K^*$ produces a reparametrization of the same curve; only the location of the identity point changes.  For part (e), note that
$\T_{aq^i}(x)=\T_a(q^{-i} x) = (-x/a)^i \T_a(x)$. Suppose $(a,b,c,p_1,\ldots,p_9)$ and $(a',b',c',p'_1,\ldots,p'_9)$ satisfy the conditions in (e).
Then
\begin{eqnarray*}
& & (a'\cdot \T_{p'_1}\T_{p'_2}\T_{p'_3}:b'\cdot \T_{p'_4}\T_{p'_5}\T_{p'_6}:c'\cdot \T_{p'_7}\T_{p'_8}\T_{p'_9})\\
&=& \biggl(\frac{(-x)^na'\T_{p_1}\T_{p_2}\T_{p_3}}{p_1^{n_1}p_2^{n_2}p_3^{n_3}}:\frac{(-x)^nb'\T_{p_4}\T_{p_5}\T_{p_6}}{p_4^{n_4}p_5^{n_5}p_6^{n_6}}:\frac{(-x)^nc'\T_{p_7}\T_{p_8}\T_{p_9}}{p_7^{n_7}p_8^{n_8}p_9^{n_9}} \biggr) \\
&=&(a\cdot \T_{p_1}\T_{p_2}\T_{p_3}:b\cdot \T_{p_4}\T_{p_5}\T_{p_6}:c\cdot \T_{p_7}\T_{p_8}\T_{p_9}),
\end{eqnarray*}
where $n = n_1+n_2+n_3=n_4+n_5+n_6=n_7+n_8+n_9$. 

Finally, for (d), one may check the identity $x \T(x^{-1})=\T(x)$ directly from the definition
of the fundamental theta function.  In light of (\ref{eq:transpara}), this implies
$$ \!  \begin{matrix}
& & (a\cdot \T_{p_1}\T_{p_2}\T_{p_3}:b\cdot \T_{p_4}\T_{p_5}\T_{p_6}:c\cdot \T_{p_7}\T_{p_8}\T_{p_9})\\
&=& \!\! (\frac{-x^3 a}{p_1p_2p_3}\T(\frac{p_1}{x})\T(\frac{p_2}{x})\T(\frac{p_3}{x}):
\frac{-x^3 b}{p_4p_5p_6}\T(\frac{p_4}{x})\T(\frac{p_5}{x})\T(\frac{p_6}{x}):\frac{-x^3 c}{p_7p_8p_9}\T(\frac{p_7}{x})\T(\frac{p_8}{x})\T(\frac{p_9}{x}))\\
&=&(a \, \T_{\! \frac{1}{p_1}}(\! \frac{1}{x})\T_{\! \frac{1}{p_2}}(\frac{1}{x})\T_{\! \frac{1}{p_3}}(\frac{1}{x}):b \,\T_{\! \frac{1}{p_4}}(\frac{1}{x})\T_{\! \frac{1}{p_5}}(\frac{1}{x})\T_{\! \frac{1}{p_6}}(\frac{1}{x}):
c\, \T_{\! \frac{1}{p_7}}(\frac{1}{x})\T_{\! \frac{1}{p_8}}(\frac{1}{x})\T_{\!\frac{1}{p_0}}(\frac{1}{x})).
\end{matrix}
$$
This is the reparametrization of the elliptic curve $E$ under the involution $x\mapsto 1/x$
in the group law.  We have thus proved the if direction of Proposition \ref{prop:whensame}.

For the only-if direction, we write $\psi$ and $\psi'$ for the maps
$E \rightarrow  C \subset \PP^2$ defined by
$(a,b,\ldots,p_9)$ and $(a',b',\ldots,p_9')$ respectively.
Then $\,(\psi')^{-1} \circ \psi\,$ is an automorphism of the elliptic curve $E$.
The $j$-invariant of $E$ is neither of the special values $0$ or $1728$.
By \cite[Theorem 10.1]{silverbaby}, the only automorphisms of $E$
are the involution $x \mapsto 1/x$ and multiplication by some fixed element in the
group law. These are precisely the operations we discussed above,
and they can be realized by the transformations
from  $(a,b,\ldots,p_9)$ to $(a',b',\ldots,p_9')$  that are
described in (c) and (d).

Finally, if $(\psi')^{-1} \circ \psi$ is the identity on  $E$,
then plugging in $x=p_i$ for $i=1,2,3$ shows that $p_i$
is a zero of $a \T_{p_1} \T_{p_2} \T_{p_3}$ and hence 
of $a' \T_{p_1'} \T_{p_2'} \T_{p_3'}$. The same holds for
$i=4,5,6$ and $i=7,8,9$. This accounts for the
operations (a), (b) and (e).
\end{proof}

Our main result in this section is the following characterization of
honeycomb curves, in terms of the analytic representation of
plane cubics  in (\ref{eq:param}).
Writing $\mathbb{S}^1$ for the  circle, let
 $\mathcal{V}:K^*\rightarrow \mathbb{S}^1$ denote the composition 
$\,K^*\stackrel{\rm val}{\rightarrow}\R\rightarrow\R/{\rm val}(q) \simeq \mathbb{S}^1$.

\begin{theorem}\label{thm:analytic2}
Let $a,b,c,p_1,\ldots,p_9\in K^*$  as in Lemma \ref{thm:analytic1}. Suppose the values 
$\mathcal{V}(p_1),\ldots,\mathcal{V}(p_9)$ occur in cyclic order on $\mathbb{S}^1$, with
$\mathcal{V}(p_3)=\mathcal{V}(p_4), \mathcal{V}(p_6) = \mathcal{V}(p_7)$,
$ \mathcal{V}(p_9)=\mathcal{V}(p_1)$, and all other values  $\mathcal{V}(p_i)$ are distinct. Then the
image of the map  (\ref{eq:param}) is an elliptic curve in honeycomb form.
Conversely, any elliptic curve in honeycomb form arises in this manner, after a suitable
permutation of the indices.
\end{theorem}

We shall present two alternative proofs of Theorem~\ref{thm:analytic2}.
These will highlight different features of honeycomb curves
and how they relate to the literature.
The first proof is computational and
relates our study to   the {\em tropical theta functions} studied by
 Mikhalkin and Zharkov \cite{mz}.
The second proof is more conceptual. It is based
on the nonarchimedean Poincar\'e-Lelong formula for {\em Berkovich curves} \cite[Theorem 5.68]{bpr}.
Both approaches were suggested to us by Matt~Baker.

\begin{proof}[First proof of Theorem \ref{thm:analytic2}]
We shall examine the naive tropicalization of the elliptic curve 
$E= K^*/q^{\mathbb{Z}}$ under its embedding (\ref{eq:param}) into  $\PP^2$.
Set $Q  = {\rm val}(q) \in \mathbb{R}_{>0}$.
If $a \in K^*$ with $A = {\rm val}(a) \in \mathbb{R}$ then
the {\em tropicalization} of the theta function  $\Theta_a : K^* \rightarrow K $ is
obtained by replacing the infinite product of binomials 
in the definition of $\Theta(x)$ by an infinite sum
of pairwise minima. The result is the function
\begin{equation}
\label{eq:troptheta1}
\, {\rm trop}(\Theta_a) : \,\mathbb{R} \,\rightarrow \,\mathbb{R} \,, \,\,\,
 X \,\mapsto \,
\sum_{n > 0} {\rm min}(0,nQ {+}X{-}A) \, + \,
\sum_{n \geq 0} {\rm min}(0,nQ{+}A{-}X).
\end{equation}
For any particular real number $X$,
only finitely many summands are non-zero,
and hence ${\rm trop}(\Theta_a)(X)$ is a 
well-defined real number. A direct calculation shows that
\begin{equation}
\label{eq:troptheta2}
{\rm trop}(\Theta_a)(X) \,\, = \,\,\,
{\rm min} \biggl\{\, \frac{m^2 - m}{2} \cdot Q + m \cdot (A-X) \,\,:\,\,
m \in \mathbb{Z} \, \biggr\}.
\end{equation}
Indeed, the distributive law transforms the
 tropical product of binomials on the right hand side of (\ref{eq:troptheta1})
into  the tropical sum in (\ref{eq:troptheta2}).
The representation (\ref{eq:troptheta2}) is essentially the same
as the {\em tropical theta function} of Mikhalkin and Zharkov~\cite{mz}.

The tropical theta function is a piecewise linear function on
$\mathbb{R}$, and we can translate (\ref{eq:troptheta2}) into 
an explicit  description of the linear pieces of its graph. We find
\begin{equation}
\label{eq:troptheta3}
{\rm trop}(\Theta_a)(X) \,\, = \,\,\,
 \frac{m^2 - m}{2} \cdot Q + m \cdot (A-X) \,\,
\end{equation} 
where $m$ is the unique integer satisfying $\,mQ \leq A-X < (m+1)Q$.
In particular, for arguments $X$ in this interval, the
function is linear with slope $-m$.

The tropical theta function approximates the
valuation of the theta function. These two functions agree unless 
there is some cancellation because the two terms in some binomial factor
of $\Theta_a$ have the same order.

The gap between the tropical theta function
and the valuation of the theta function 
is crucial in understanding the tropical geometry of the map
$\mathcal{V} : K^* \rightarrow \mathbb{S}^1$. Our next definition makes this precise. If $x,y \in K^*$
with $\mathcal{V}(x) = \mathcal{V}(y) $ then we~set
\begin{equation}
\label{eq:deltadef}
\delta(x,y) \,\,:=\,\, {\rm val}\bigl( 1 - \frac{x}{y} q^i \bigr) 
\end{equation}
where $i \in \mathbb{Z}$ is the unique integer satisfying
${\rm val}(x) + {\rm val}(q^i) = {\rm val}(y)$.
It is easy to check that the quantity defined in 
(\ref{eq:deltadef}) is symmetric, i.e.
$\,\delta(x,y) \,\, = \,\, \delta(y,x) $.
With this notation, the following formula
characterizes the gap  between the tropical theta function
and the valuation of the theta function.
For any $a,x \in K^*$, we have
\begin{equation}
\label{eq:thetagap}
{\rm val} \bigl(\Theta_a(x) \bigr) -
{\rm trop}(\Theta_a)({\rm val}(x))  \,\,\, = \,\,\,
\begin{cases}
\,\delta(x,a) & {\rm if} \, \mathcal{V}(a) = \mathcal{V}(x), \\
\quad \,\,0 & \hbox{otherwise.}
\end{cases} 
\end{equation}
 
Consider now any three scalars $x,y,z \in K^*$ that lie in the same
fiber of the map from $K^*$ onto the unit circle
$\mathbb{S}^1$. In symbols, $\mathcal{V}(x) = \mathcal{V}(y) = \mathcal{V}(z)$.
Then 
\begin{equation}
\label{eq:deltatriple}
 \hbox{the minimum of $\delta(x,y)$, $ \delta(x,y) $
and $\delta(y,z)$ is attained twice.} 
\end{equation}
This follows from the identity
$$ (xq^i - yq^j) + (yq^j-z) + (z - x q^i) \,\,= \,\, 0 $$
where  $i,j \in \mathbb{Z}$ are defined by
${\rm val}(x) + {\rm val}(q^i) = {\rm val}(y) + {\rm val}(q^j)= {\rm val}(z) $.

We are now prepared to prove Theorem \ref{thm:analytic2}.
Set $Q = {\rm val}(q)$.
For $i=1,2,\ldots,9$, let $P_i = {\rm val}(p_i)$
and write $P_i = n_i Q + r_i$ where $n_i \in \mathbb{Z}$ and $r_i \in [0,Q)$.
Rescaling the $p_i$'s by a common factor and inverting them all does not
change the cubic curve, by Proposition  \ref{prop:whensame}.
After performing such operations if needed, we can assume
\begin{equation}
\label{eq:rrrQ}
 0 = r_9 = r_1 < r_2 < r_3 = r_4 < r_5 < r_6 = r_7 < r_8 < Q. 
 \end{equation}
The hypothesis (\ref{eq:transpara}) implies
$$ r_1 + r_2 + r_3 \,\,\equiv \,\,
r_4 + r_5 + r_6 \,\, \equiv \,\,
r_7 + r_8 + r_9 \qquad ({\rm mod} \,\,Q). $$
Together with the chain of inequalities in (\ref{eq:rrrQ}), this implies
\begin{equation}
\label{eq:niceshift}
\begin{matrix}  r_1 + r_2 + r_3 & = &
r_4 + r_5 + r_6 - Q & = &
r_7 + r_8 + r_9 -Q ,\\
 n_1 + n_2 + n_3 & = &
n_4 + n_5 + n_6 +1 & = &
n_7 + n_8 + n_9 +1 .
\end{matrix} 
\end{equation}

We now examine the naive tropicalization
$\mathbb{R} \rightarrow \mathbb{TP}^2$ of our
map (\ref{eq:param}). It equals
\begin{equation}
\label{eq:tropparam}
 X \,\,\mapsto \,\,
\begin{matrix} \bigl(\,
A + 
{\rm trop}(\Theta_{p_1})(X) + 
{\rm trop}(\Theta_{p_2})(X) + 
{\rm trop}(\Theta_{p_3})(X) : \\ \,\,\,
B + 
{\rm trop}(\Theta_{p_4})(X) + 
{\rm trop}(\Theta_{p_5})(X) + 
{\rm trop}(\Theta_{p_6})(X) : \\ \,\,\,\,
C + 
{\rm trop}(\Theta_{p_7})(X) + 
{\rm trop}(\Theta_{p_8})(X) + 
{\rm trop}(\Theta_{p_9})(X) \,\,\bigr). 
\end{matrix}
\end{equation}
Here $A = {\rm val}(a), B = {\rm val}(b), C = {\rm val}(c)$.
The image of this piecewise-linear map coincides with the
tropicalization of the image of (\ref{eq:param}) for 
generic values of $X = {\rm val}(x)$. Indeed, if $X$ lies
in the open interval $(0,r_2)$ then, by  (\ref{eq:troptheta3}),
the map (\ref{eq:tropparam}) is linear
and its image in the tropical projective plane $\mathbb{TP}^2$ is a segment with slope
$$ (-n_1-n_2-n_3-1 : -n_4-n_5-n_6 : n_7-n_8-n_9-1)
\,\,=\,\, (\,1\, : \,1\, : \,0\,). $$
Here we fix tropical affine coordinates with last coordinate $0$. Similarly,
\begin{itemize}
\item the slope is $(0 : 1 : 0)$ for $X \in (r_2,r_3)$,
\item the slope is $(-1 : 0 : 0)$ for $X \in (r_4,r_5)$,
\item the slope is $( -1 : -1 : 0 )$ for $X \in (r_5,r_6)$,
\item the slope is $( 0 : -1 : 0 )$ for $X \in (r_7,r_8)$,
\item the slope is $(1 : 0 : 0)$ for $X \in (r_8,Q)$.
\end{itemize}
These six line segments form a hexagon in  $\mathbb{T}\mathbb{P}^2$.
The vertices of that hexagon are the images 
of the six distinct real numbers $r_i$ in (\ref{eq:rrrQ}) under the map (\ref{eq:tropparam}).

Finally, we  examine the special values $x \in K^*$
for which the naive tropicalization (\ref{eq:tropparam})
does not compute the correct image in $\mathbb{T}\mathbb{P}^2$.
This happens precisely when some of the nine theta functions in
(\ref{eq:param}) have a valuation gap when passing to
(\ref{eq:tropparam}).
 
Suppose, for instance, that $\mathcal{V}(x) = \mathcal{V}(p_1) = \mathcal{V}(p_9)$.
From (\ref{eq:thetagap}) we~have
\begin{equation}
\label{eq:thetagap2}
{\rm val} \bigl( \T_{p_i} (x) \bigr) -
{\rm trop}( \T_{p_i} )( {\rm val} (x) )  \,\,\, = \,\,\,
\begin{cases}
\,\delta(x,p_i) & {\rm if} \,\, i = 1 \,\,{\rm or} \,\, i=9, \\
\qquad 0 & {\rm if} \,\, i \not= 1,9.
\end{cases} 
\end{equation}
We know from (\ref{eq:deltatriple}) that the minimum 
of $\delta(p_1,p_9),\delta(x,p_1), \delta(x,p_9)$ is achieved twice.
Moreover, by varying the choice of the scalar $x$ with $\mathcal{V}(x) = r_1$, the latter two quantities can 
attain any non-negative value that is compatible with this constraint.
This shows that the image of the set of
such $x$ under the tropicalization of the map (\ref{eq:param})
consists of one bounded segment and two rays in $\mathbb{T}\mathbb{P}^2$.
The segment meets the hexagon described above
at the vertex corresponding to $r_1 = r_9$, and consists of the images of
the points $x$ such that $\delta(p_1, p_9)\ge \delta(p_1, x)=\delta(p_9,x)$.
Since $\Theta_{p_1}$ and $\Theta_{p_9}$ occur in the first and third
coordinates, respectively, of (\ref{eq:param}), the slope of the  segment is
$(1:0:1)$, and its length is $\delta(p_1,p_9)$.  Similarly, the
image of the points $x$ such that $\delta(p_1,x)\ge
\delta(p_1,p_9)=\delta(p_9,x)$
is a ray of slope $(1:0:0)$, and the image of the points $x$ such that
$\delta(p_9,x)\ge \delta(p_1,p_9)=\delta(p_1,x)$
is a ray of slope $(0:0:1)$.
Note that these three slopes obey  the balancing condition for tropical curves.

A similar analysis determines all six connected components of
the complement of the hexagon in the tropical curve, and we
see that the tropical curve is a honeycomb when
the asserted conditions on $a,b,c,p_1,p_2,\ldots,p_9$ are satisfied.

The derivation of the converse direction, that any
honeycomb cubic has the desired parametrization,
will be deferred to the second proof.
It seems challenging to  prove this without \cite{bpr}.
See also the problem stated at the end of this section.
\end{proof}

\begin{figure}
\includegraphics[scale=0.7]{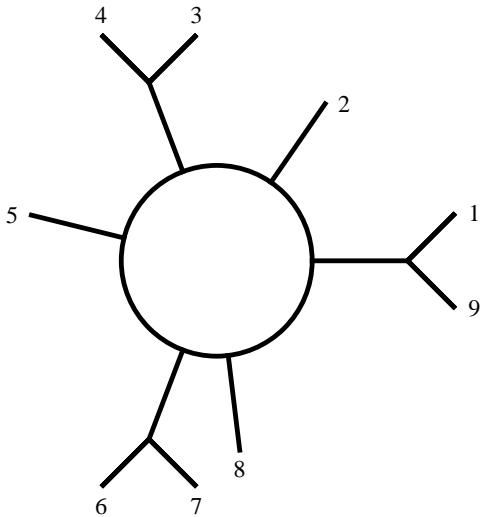}
\vskip -0.1in
\caption{The Berkovich skeleton $\Sigma$ of an elliptic curve with honeycomb punctures}
\label{fig:zwei}
\end{figure}

\begin{proof}[Second proof of Theorem \ref{thm:analytic2}]
We work in the setting of Berkovich curves introduced by Baker, Payne and Rabinoff in \cite{bpr}.
Let $E^{an}$ denote the {\em analytification} of the elliptic curve $E$,
and let $\Sigma$ denote the minimal skeleton of $E^{an}$,
as defined in \cite[\S 5.14]{bpr},  with respect to the given
set $\,D = \{p_1,p_2,\ldots,p_9 \}$ of nine punctures.
Our standing assumption ${\rm val}(j) < 0$ ensures that the Berkovich curve
$E^{an}$ contains a unique cycle $\mathbb{S}^1$, and 
$\Sigma$ is obtained from that cycle by attaching trees with nine leaves in total.
In close analogy to \cite[\S 7.1]{bpr}, we consider the retraction map
onto~$\mathbb{S}^1$:
$$ E(K)\backslash D \, \hookrightarrow \,E^{an}\backslash D \,\rightarrow \,\Sigma \,\rightarrow\,
 \mathbb{R}/{\val(q)\mathbb{Z}} \,\simeq \,\mathbb{S}^1. $$
 The condition in Theorem \ref{thm:analytic2} states that,
under this map, the points of $E(K)$ given by $p_1,p_2,\ldots,p_9$ retract to six
distinct points on $\mathbb{S}^1$, in cyclic order with
fibers $\{p_9,p_1\}, \{p_2\},\{p_3,p_4\}, \{p_5\},\{p_6,p_7\},\{p_8\}$.
This means that $\Sigma$ looks precisely like the graph in Figure \ref{fig:zwei}.
This picture is the Berkovich model of a honeycomb cubic.
To see this, we shall apply the nonarchimedean Poincar\'e-Lelong Formula
\cite[Theorem 5.69]{bpr} in conjunction with a combinatorial argument
about balanced graphs in $\mathbb{R}^2$.

The rational functions $f$ and $g$ in (\ref{eq:f_and_g}) are
well-defined on $E^{an} \backslash D$, and
we may consider the negated logarithms of their norms:
$F = -{\rm log} |f|$ and $G = -{\rm log}|g|$.
Our tropical curve can be identified with its image 
in $\R^2$ under the map $(F,G)$.
According to part (1) of \cite[Theorem 5.69]{bpr}, this map
factors through the retraction of $E^{an} \backslash D$ onto $\Sigma$. By part (2), 
the function $(F,G)$ is  linear on each edge of $\Sigma$.

We shall argue that the graph $\Sigma$ in
Figure \ref{fig:zwei} is mapped isometrically onto a  tropical honeycomb curve in $\R^2$.
Using part (5) of \cite[Theorem 5.69]{bpr}, we can determine the slopes
of the nine unbounded edges. Namely, $F$ has slope $1$ on the rays in $\Sigma$ towards $p_1, p_2$ and $p_3$, and slope $-1$ on the rays towards $p_7, p_8 $ and $p_9$. 
Similarly, $G$ has slope $1$ on the rays in $\Sigma$ towards $p_4, p_5$ and $p_6$, and slope $-1$ on the rays towards $p_7, p_8$ and $p_9$.  
By part (4), the functions $F$ and $G$ are harmonic, which means that the image
in $\R^2$ satisfies the balancing condition of tropical geometry.
This requirement uniquely determines the slopes of the nine bounded edges
in the image of $\Sigma$. For the three edges not on the cycle this is immediate,
and for the six edges on the cycle, this follows by solving a linear system of equations.
The unique solution to these constraints is a balanced planar graph
that must be a honeycomb cubic. Conversely, every
tropical honeycomb cubic 
in $\mathbb{R}^2$ is trivalent with all multiplicities one.
By \cite[Corollary 6.27(1)]{bpr}, the skeleton
$\Sigma$ must look like Figure \ref{fig:zwei},
and the corresponding map $(F,G)$ is an isometry onto the cubic.
\end{proof}    

In the rest of Section 3, we discuss
computational aspects of the representation     
of plane cubics given in Lemma  \ref{thm:analytic1} 
and Theorem \ref{thm:analytic2}. We begin
with the {\em implicitization problem}: Given $a,b,c,p_1,\ldots,p_9 \in K^*$,
how can we compute the implicit equation (\ref{eq:cubic})?
Write  $\bigl(f(x): g(x):h(x) \bigr)$ for the analytic 
parametrization  in (\ref{eq:param}). Then we seek to compute
the unique (up to scaling) coefficients $c_{ijk}$ in 
a $K$-linear relation
\begin{equation}
\label{eq:ansatz}
\begin{matrix}
 c_{300}f(x)^3+c_{210}f(x)^2 g(x) +c_{120}f(x) g(x)^2 \qquad \qquad \qquad \qquad  \\
 +\, c_{030}g(x)^3+c_{021}g(x)^2h(x) + c_{012}g(x) h(x)^2 
   +c_{003}h(x)^3 \qquad \quad \, \\
   +\,c_{102}f(x)h(x)^2+c_{201}f(x)^2 h(x)+c_{111}f(x)g(x)h(x) \qquad = \quad 0 .
\end{matrix}
\end{equation}
Evaluating this relation at $x=p_i$, and noting that
$\T_{p_i}(p_i) = 0$, we get nine linear equations for the nine $c_{ijk}$'s
other than $c_{111}$. These equations are
$$ 
\begin{matrix}
c_{300}f(p_i)^3+c_{210}f(p_i)^2 g(p_i) +c_{120}f(p_i) g(p_i)^2  + c_{030}g(p_i)^3 \,\,=\,\,0
 &{\rm for} \,\,\, i=7,8,9  , \\
c_{003}h(p_i)^3 + c_{102}f(p_i)h(p_i)^2+c_{201}f(p_i)^2 h(p_i)+   c_{300}f(p_i)^3
\,\,=\,\,0  &    {\rm for} \,\,\, i=4,5,6 ,\\
 c_{030}g(p_i)^3+c_{021}g(p_i)^2h(p_i) + c_{012}g(p_i) h(p_i)^2 
   +c_{003}h(p_i)^3 \,\,=\,\,0
    & {\rm for} \,\,\, i=1,2,3 .
  \end{matrix}   $$
   The first group of equations has a solution
   $(c_{300}, c_{210}, c_{120}, c_{030})$ 
   that is unique up to scaling. Namely,
   the ratios $c_{300}/c_{030}, c_{210}/c_{030}, c_{120}/c_{030}$
   are the elementary symmetry functions in the three quantities
   $$
   \frac{b \cdot   \T(\frac{p_7}{p_4}) \T(\frac{p_7}{p_5}) \T(\frac{p_7}{p_6})}
{a   \cdot \T(\frac{p_7}{p_1}) \T(\frac{p_7}{p_2}) \T(\frac{p_7}{p_3})}\, ,\,\,\,
  \frac{b \cdot   \T(\frac{p_8}{p_4}) \T(\frac{p_8}{p_5}) \T(\frac{p_8}{p_6})}
{a   \cdot \T(\frac{p_8}{p_1}) \T(\frac{p_8}{p_2}) \T(\frac{p_8}{p_3})} \quad \hbox{and} \quad
  \frac{b \cdot   \T(\frac{p_9}{p_4}) \T(\frac{p_9}{p_5}) \T(\frac{p_9}{p_6})}
{a   \cdot \T(\frac{p_9}{p_1}) \T(\frac{p_9}{p_2}) \T(\frac{p_9}{p_3})} .
$$
   The analogous statements hold for the second and third group of equations.

We are thus left with computing the
middle coefficient $c_{111}$ in the relation
(\ref{eq:ansatz}). We do this by
picking any $v \in K$ with
$f(v) g(v) h(v) \not= 0$.
Then (\ref{eq:ansatz}) gives
$$ c_{111} \,\,= \,\, - \frac{1}{f(v) g(v) h(v)} \bigl(
 c_{300}f(v)^3+c_{210}f(v)^2 g(v) + \cdots +c_{201}f(v)^2 h(v)\bigr). $$
We have implemented this implicitization method in \textsc{Mathematica},
for input data in the field $K = \mathbb{Q}(t)$ of rational functions
with rational coefficients.

The {\em parameterization problem} is harder. 
Here we are given the $10$ coefficients $c_{ijk}$
of a honeycomb cubic that has three distinct intersection
points with each coordinate line $z_0=0, z_1=0, z_2=0$.
The  task is to compute $12$ scalars $a,b,c,p_1,\ldots,p_9 \in K^*$
that represent the cubic as in Lemma \ref{thm:analytic1}.
The $12$ output scalars are not unique, but the degree of
non-uniqueness is characterized exactly by
Proposition \ref{prop:whensame}.
This task amounts to solving an analytic system of equations.
We shall leave it to a future project 
to design an algorithm for doing this in practice.

\section{The Tropical Group Law}

In this section, we present a combinatorial description of
the group law on a honeycomb elliptic curve based on the
parametric representation  in Section~3.  
We start by studying the inflection points of such a curve.
We continue to assume that
$K$ is  algebraically closed and complete with respect to a nonarchimedean valuation.  

Let $\,E \buildrel{\psi}\over{\longrightarrow} C \subset \PP^2\,$
be a honeycomb embedding of  the abstract elliptic curve $E= K^*/q^{\mathbb{Z}}$.
Let $v_1,v_2,v_3,v_4,v_5$ and $v_6$ denote the vertices of the hexagon in the tropical cubic ${\rm trop}(C)$, labeled as in Figure \ref{fig:law1}.
Let $e_i$ denote the edge between
$v_i$ and $v_{i+1}$, with the convention $v_7 = v_1$,
and let $\ell_i$ denote the lattice length of $e_i$.
  By examining the width and height of the hexagon, we see that the six lattice lengths
 $\ell_1,\ell_2,\ell_3,\ell_4,\ell_5,\ell_6$
  satisfy two linearly independent  relations:
  $$ \ell_1 + \ell_2 \, = \ell_4 + \ell_5 \quad \hbox{and} \quad
      \ell_2+\ell_3 \, = \ell_5 + \ell_6 . 
      $$
We first prove the following basic fact about 
the inflection points on the cubic~$C$.

\begin{figure}
\includegraphics[scale=0.7]{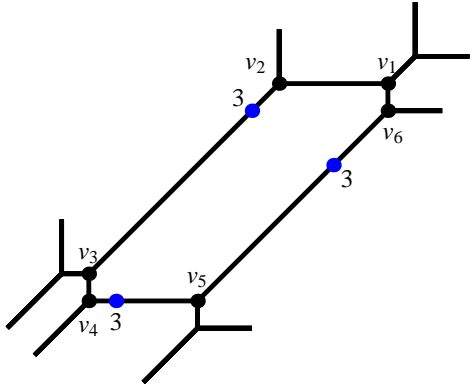}
\vskip -0.1in
\caption{A honeycomb cubic and its nine inflection points in groups of three}
\label{fig:law1}
\end{figure}

\begin{lemma}\label{l:inflect}
The tropicalizations of the nine inflection points of
the cubic curve $C \subset \mathbb{P}^2$ retract to the 
hexagonal cycle 
of ${\rm trop}(C) \subset \mathbb{R}^2$ in three groups of~three.
\end{lemma}

\begin{proof}
This lemma is best understood from the perspective of Berkovich theory.
The analytification $E^{\rm an}$ retracts onto its skeleton, namely the unique
cycle, which is isometrically embedded into ${\rm trop}(C)$ as the hexagon.
Thus every point of $E$ retracts onto a unique point in the hexagon.  In fact,
this retraction is given by
\begin{equation}
\label{eq:retract}
K^*/q^{\mathbb{Z}} \,\cong \,E(K)\,\hookrightarrow \, E^{\rm an}
\, \twoheadrightarrow \ \mathbb{S}^1\,\cong \,\R/{\rm val}(q) \Z
\end{equation}
and is the natural map induced from the valuation homomorphism
$\,(K^*, \,\cdot\,) \rightarrow (\R,+)$.  We refer the reader unfamiliar with the map (\ref{eq:retract}) to \cite[Theorem 7.2]{bpr}.

Now, after a multiplicative translation, we may assume that $\psi$ sends the identity of $E$ to an inflection point.  Then the inflection points of $C$ are the images of the 3-torsion points of $E = K^*/q^{\mathbb{Z}}$.   These can be written as $\omega^i \cdot q^{j/3}$ for $\omega^3=1$, for $q^{1/3}$ a cube root of $q$, and for $i,j=0,1,2$.  Note that ${\rm val}(\omega) = 0$ whereas ${\rm val}(q^{1/3}) = {\rm val}(q)/3>0$.
 Hence the valuations of the scalars 
  $\omega^i \cdot q^{j/3}$ are $0$, ${\rm val}(q)/3$, and
  $2 {\rm val}(q)/3$, and each group contains
three of these nine scalars.
\end{proof}

Our next result refines Lemma~\ref{l:inflect}. It is a very special case of a theorem due to
Brugall\'e and de Medrano~\cite{BruMed} which covers
honeycomb curves of arbitrary degree.

\begin{lemma}
Let $P \in {\rm trop}(C)$ be the tropicalization of an inflection point on the 
cubic curve $C \subset \mathbb{P}^2$. Then there are three possibilities,
as indicated in Figure \ref{fig:law1}:
\begin{itemize}
\item  The point $P$ lies on the longer of $e_1$ or $e_2$, at distance
$|\ell_2-\ell_1| /3$ from $v_2$.
\item  The point $P$ lies on the longer of $e_3$ or $e_4$, at distance
$|\ell_4-\ell_3| /3$ from $v_4$.
\item  The point $P$ lies on the longer of $e_5$ or $e_6$, at distance
$|\ell_6-\ell_5| /3$ from $v_6$.
\end{itemize}
The nine inflection points fall into three groups of three in this way.
\end{lemma}
In the special case that $\ell_1= \ell_2$ (and similarly 
$\ell_3= \ell_4$ or $\ell_5=\ell_6$), the lemma should be understood as saying that $P$ lies somewhere on the ray emanating from~$v_2$.

\begin{proof}
Consider the tropical line
whose node lies at $v_2$, and let $L$ be any classical line in $\PP^2_K$
that is generic among lifts of that tropical line. Then the three points
in $L \cap C$ tropicalize to the stable intersection points of
${\rm trop}(L)$ with ${\rm trop}(C)$, namely the vertices
$v_1,v_2$ and $v_3$.
Let $x$ denote the counterclockwise distance along the hexagon from $v_2$ to $P$.
By applying a multiplicative translation in $E$, 
we fix the identity to be an inflection point on $C$
that tropicalizes to $P$. Then
 $v_1 + v_2 + v_3 = 0$ in the group $\mathbb{S}^1 \simeq \R/{\rm val}(q) \Z$.
This observation implies the congruence relation
$$ (\ell_1 + x) + x + (x - \ell_2) \,\, \equiv \,\, 0 \quad {\rm mod}\,\, \,{\rm val}(q), $$
and hence $\,3x \equiv \ell_2 - \ell_1 $.
One solution to this congruence is $x = (\ell_2 - \ell_1)/3$.
This means that the point $P$ lies on the longer edge, $e_1$ or $e_2$,
at distance $|\ell_1-\ell_2| /3$ from $v_2$. The analysis is identical for 
the vertex $v_4$ and for the vertex $v_6$, and it identifies two other
locations on $\mathbb{S}^1$ for retractions of inflection points.
\end{proof}

Next, we wish to review some basic facts about the
group structures on a plane cubic curve $C$.
The abstract elliptic curve $E= K^*/q^{\mathbb{Z}}$
is an abelian group, in the obvious sense, as a quotient
of the abelian group $(K^*, \,\cdot \,)$. Knowledge of that group structure
is equivalent to knowing the surface $\{ (r,s,t) \in E^3 : r \cdot s \cdot t 
= q^{\mathbb{Z}} \}$.

The group structure on $E$ induces a group structure  $(C, \,\star \,)$
on the plane cubic curve $C = \psi(E)$.
While the isomorphism $\psi$ is analytic, 
the group operation on $C$ is actually algebraic. Equivalently,
the set
$\{(u,v,w) \in C^3 : u \star v \star w = {\rm id} \}$ is an
algebraic surface in $(\mathbb{P}^2)^3$.
However, this surface depends heavily 
on the choice of parametrization $\psi$.
Different isomorphisms from 
the abstract elliptic curve $E$ onto the plane cubic $C$ will result
in different group structures in $\mathbb{P}^2$.

The most convenient  choices of $\psi$
are those that send the identity element $  q^{\mathbb{Z}} $ of $E$
to one of the nine inflection points of $C$. 
Such maps $\psi$ are characterized by the condition that 
$\,u \star v \star w = {\rm id}  \,$ if and only if
$u,v$ and $w$ are collinear in $\mathbb{P}^2$.
If so, the data of the group law can be recorded in the surface
\begin{equation}
\label{eq:gplaw1}
\bigl\{(u,v,w) \in C^3 \,: \,\hbox{$u$, $v$ and $w$ lie on a line in $\mathbb{P}^2$} \,\bigr\}
\quad \subset \quad (\mathbb{P}^2)^3.
\end{equation}

We emphasize, however, that
we have no reason to assume a priori that the identity
element $\,{\rm id} = \psi( q^{\mathbb{Z}} )\,$ on
the plane cubic $C$ is an inflection point:
{\em every point on $C$ is the identity in some group law}.
Indeed, we can simply replace $\psi$ 
by its composition with 
a translation
$\,x \mapsto r \cdot x \,$ of the group $E$.
Our analysis below covers all cases:
our combinatorial description
of the group law on ${\rm trop}(C)$ is always valid,
regardless of whether the identity is an 
inflection point or~not.

A partial description of the tropical group law
was given by  Vigeland in \cite{vige}.
Specifically, his choice of the fixed point
$\,\mathcal{O}\,$ corresponds to the point
$\,{\rm trop}(\psi( q^{\mathbb{Z}} ))$.
Vigeland's group law is a combinatorial extension of the
polyhedral surface
\begin{equation}
\label{eq:tgplaw}
\bigl\{ (U,V,W) \in \mathbb{S}^1 \times \mathbb{S}^1 \times \mathbb{S}^1\,:\,
U+V+W = \mathcal{O} \bigr\} 
\end{equation}
This torus comes with a distinguished subdivision into polygons
since $\mathbb{S}^1$ has been identified with the hexagon.
The polyhedral torus (\ref{eq:tgplaw}) is illustrated in Figure \ref{fig:torus}.

\begin{figure}
\centering
\subfloat{\includegraphics[width=0.4\textwidth]{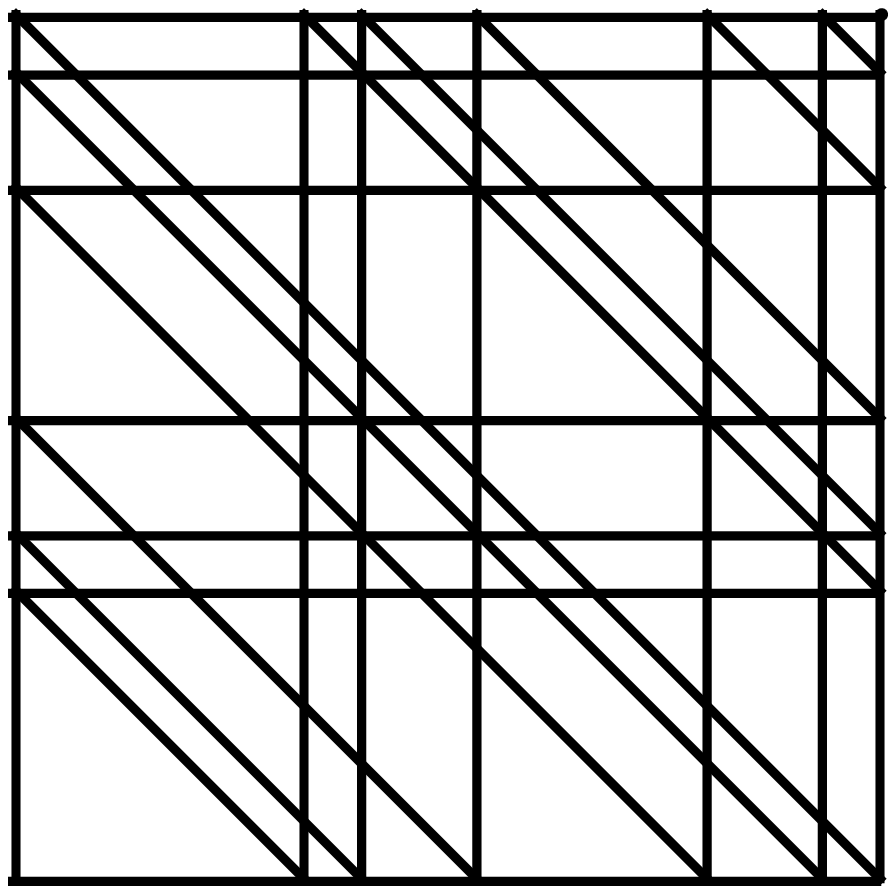}} 
\qquad\qquad
\subfloat{\includegraphics[width=0.4\textwidth]{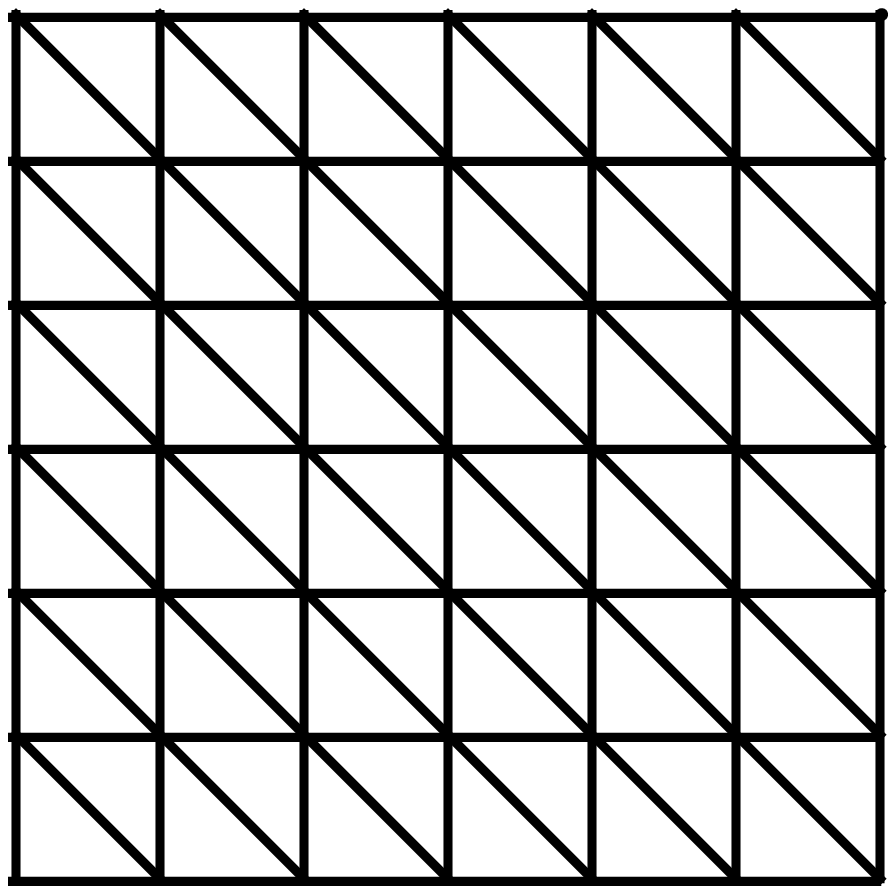}} 
\vskip -0.1in
\caption{The torus in the tropical group law surface
(\ref{eq:tgplaw}). The two pictures represent
 the honeycomb cubic and the symmetric honeycomb cubic shown in Figure~\ref{fig:eins}.}
\label{fig:torus}
\end{figure}
\smallskip

Our goal is to characterize the tropical group law as follows.
 For a given honeycomb embedding $\psi:E\rightarrow  \mathbb{P}^2$, we define the {\em tropical group law surface} to be
$${\rm TGL}(\psi) = \bigl\{({\rm trop}(\psi(x)),{\rm trop}(\psi(y)),{\rm trop}(\psi(z))): x,y,z\in E, x\cdot y\cdot z = {\rm id}
\bigr\}\subset (\mathbb{R}^2)^3.$$
If $\psi$ sends the identity to an inflection point, 
then this is the tropicalization of (\ref{eq:gplaw1}):
$$ {\rm TGL}(\psi) \,=\,
\{({\rm trop}(u),{\rm trop}(v),{\rm trop}(w)): \hbox{$u,v,w\in C$ lie on a line in $\mathbb{P}^2$}\}\subset (\mathbb{R}^2)^3.$$
The tropical group law surface is a tropical algebraic variety of dimension $2$, 
and can in principle be computed, for $K = \mathbb{Q}(t)$, using the software $\texttt{gfan}$ \cite{gfan}.
This surface contains all the information of the {\em tropical group law}.  
We shall explain how ${\rm TGL}(\psi)$ can be computed combinatorially,
even if $\psi({\rm id}) $ is not an inflection point.

Our approach  follows directly from Section 3.  Let $\psi: E \rightarrow \mathbb{P}^2$ be a honeycomb embedding of an elliptic curve $E$, and let
 $V:K^*\rightarrow \mathbb{S}^1$ denote the composition 
$\,K^*\stackrel{\rm val}{\rightarrow}\R\rightarrow\R/{\rm val}(q) \simeq \mathbb{S}^1$.
By Theorem~\ref{thm:analytic2}, there exist $a,b,c,p_1,\ldots,p_9\in K^*$ with 
$\mathcal{V}(p_1),\ldots,\mathcal{V}(p_9)$ occurring in cyclic order on $\mathbb{S}^1$, with
$\mathcal{V}(p_3) = \mathcal{V}(p_4), \mathcal{V}(p_6) = \mathcal{V}(p_7), \mathcal{V}(p_9)=\mathcal{V}(p_1)$, and all other values  $\mathcal{V}(p_i)$ distinct, such that
$\psi$ is given~by
\begin{equation*}
x \mapsto
\bigl( a\cdot \T_{p_1}\T_{p_2}\T_{p_3}(x)
:b\cdot \T_{p_4}\T_{p_5}\T_{p_6}(x)
:c\cdot \T_{p_7}\T_{p_8}\T_{p_9}(x)  \bigr).
\end{equation*}

Now, again as in Section 3, given $x,y \in K^*$
with $\mathcal{V}(x) = \mathcal{V}(y) $, we~set
\begin{equation*}
\delta(x,y) \,\,=\,\, {\rm val}\bigl( 1 - \frac{x}{y} q^i \bigr), 
\end{equation*}
where $i \in \mathbb{Z}$ is specified by
${\rm val}(x) + {\rm val}(q^i) = {\rm val}(y)$.
We have seen that if $x\in K^*$ satisfies $\mathcal{V}(x) = \mathcal{V}(p_i)$, 
then ${\rm trop}(x)$ lies at distance $\delta(x,p_i)$ from the hexagon 
along the tentacle associated to $p_i$.  
By the {\em tentacle} of $p_i$ we mean the union of the ray associated to $p_i$ 
and the bounded segment to which that ray is attached.

The next proposition is our main result in Section 4.
We shall construct the tropical group law surface by way
of its projection to the first two coordinates
$$\pi: {\rm TGL}(\psi) \rightarrow {\rm trop}(C)\times{\rm trop}(C) \,\subset \, (\mathbb{R}^2)^2 .$$ 
As before, we identify   the hexagon of ${\rm trop}(C) $
with the circle $\,\mathbb{S}^1 \simeq \R/{\rm val}(q) $.
 Given $U,V\in {\rm trop}(C)$,  let $U\circ V\in \mathbb{S}^1$ denote the sum of the retractions of 
 $U$ and $V$ to the hexagon.
  The location of $U\circ V$ depends on the choice of $\psi$ and in particular the location of 
  $\,\mathcal{O} = {\rm trop}(\psi({\rm id}))$ on 
  the hexagon $\mathbb{S}^1$.
     Let $-(U\circ V)$ denote the inverse of $U\circ V$, again under addition on 
     $\mathbb{S}^1$.  By the {\em distance} of a point $U\in {\rm trop}(C)$ to 
     the hexagon $\mathbb{S}^1$ we mean the lattice length of the unique path in ${\rm trop}(C)$ from $U$ to
      $\mathbb{S}^1$. 
      Finally, we say that $u\in K^*$ is a lift of $U \in {\rm trop}(C)$ if ${\rm trop}(\psi(u \cdot q^\mathbb{Z}))=U$.  
The following proposition characterizes the fiber of the map $\pi$ over a given pair $(U,V)$.

\begin{proposition}\label{p:fibers}
Let $\psi: E \rightarrow C\subset \mathbb{P}^2$ be a honeycomb embedding, with
  the operation $\circ:{\rm trop}(C)\times {\rm trop}(C)\rightarrow \mathbb{S}^1$ defined as above,
  and ``is a vertex'' refers to the hexagon $\mathbb{S}^1$.    For any 
  $\,U$ and $V\in {\rm trop}(C)$, exactly one of the following occurs:
\begin{enumerate}
	\item If $-(U\circ V)$ is not a vertex, then $\pi^{-1}(U,V)$ is the singleton $\{-(U\circ V)\}.$
	\item If $-(U\circ V)$ is a vertex adjacent to a single unbounded ray $R_i$
	 towards the point $p_i$,  then $\pi^{-1}(U,V)$ is the set of points on $R_i$ whose distance to $\mathbb{S}^1$ equals $\delta(u^{-1}v^{-1},p_i)$ for some lifts $u,v\in K^*$ of $\,U$ and $V$. 
	\item If $-(U\circ V)$ is a vertex adjacent to a bounded segment $B$, along with two rays $R_j$ and $R_k$,
	toward the points $p_j$ and $p_k$, then $\pi^{-1}(U,V)$ consists of
		\begin{itemize}
	\item the points on $B$ whose distance to $\mathbb{S}^1$ is equal to $\delta(u^{-1}v^{-1},p_j)\,=\delta(u^{-1}v^{-1},p_k)\,$ for some lifts $u,v\in K^*$ of $\,U$ and $V$,
	\item the points on $R_j$ whose distance to $\mathbb{S}^1$ is equal to 
	$\delta(u^{-1}v^{-1},p_j) >  \delta(u^{-1}v^{-1},p_k)\,$ for some lifts $u,v\in K^*$ of $\,U$ and $V$, and
	\item the points on $R_k$ whose distance to $\mathbb{S}^1$ is equal to 
	$\delta(u^{-1}v^{-1},p_k)  > \delta(u^{-1}v^{-1},p_j)\,$ for some lifts $u,v\in K^*$ of $\, U$ and $V$.
		\end{itemize}
\end{enumerate}
\end{proposition}

\begin{proof} For any lifts $u,v\in K^*$ of $U,V$, we have 
$\, \mathcal{V}(u^{-1}v^{-1}) + \mathcal{V}(u) + \mathcal{V}(v) = 0 \,$ in $\,\mathbb{S}^1 = \R/{\rm val}(q) \Z$.
This equation determines the retraction $-(U\circ V)$  of 
the point ${\rm trop}(\psi(u^{-1}v^{-1}))$ to the hexagon.  
If $-(U\circ V)$ is not a vertex, then ${\rm trop}(\psi(u^{-1}v^{-1}))$ must be precisely the point 
$-(U\circ V)$, as no other points retract to it.  

If instead $-(U\circ V)$ is a vertex of the hexagon, then we either have $\mathcal{V}(u^{-1}v^{-1})=\mathcal{V}(p_i)$ for exactly one $p_i$, or $\mathcal{V}(u^{-1}v^{-1})=\mathcal{V}(p_j)=\mathcal{V}(p_k)$ for exactly two points $p_j$ and $p_k$, depending on whether one or two rays emanate from $-(U\circ V)$.  
We have seen (in our first proof of Theorem~\ref{thm:analytic2}) that the distance of ${\rm trop}(\psi(u^{-1}v^{-1}))$ to the hexagon, measured along the tentacle associated to $p_i$, is $\delta(u^{-1}v^{-1}, p_i)$.
In either case, these distances uniquely determine the location of $\,{\rm trop}(\psi(u^{-1}v^{-1}))$. 
\end{proof}

We demonstrate this method in a special example which was  found in discussion with Spencer Backman.
Pick $r,s\in K^*$ with $r^6=q\,$ and $\,{\rm val}(1-s)=:\beta>0$.  Let 
\begin{equation}
\label{eq:ps}
\begin{matrix}
p_1 = r^{-1}s^{-1},  &  p_2 = 1,  & p_3 = rs , \\
p_4 = rs^{-1},   & p_5 = r^2 , & p_6 = r^{-3}s, \\
p_7 = r^{3}s^{-1},  & p_8 = r^{-2},  & p_9 = r^{-1}s.
\end{matrix}
\end{equation}
 We also set $a=b=c=1$ in (\ref{eq:param}).
This choice produces a symmetric honeycomb embedding $\psi:E\rightarrow C\subset \mathbb{P}^2$.
The parameter $\beta$ is the common length of the three bounded segments
adjacent to $\mathbb{S}^1$.
Note that $p_1p_2p_3 = p_4p_5p_6 = p_7p_8p_9 = 1 $ in $K^*$.
This condition  implies that the identity of $E$ is mapped to an inflection point in $C$.

\begin{corollary} \label{cor:eleven}
For the elliptic curve in honeycomb form defined by
 (\ref{eq:ps}), the tropical group law ${\rm TGL}(\psi)$ is a polyhedral surface consisting of
117 vertices, 279 bounded edges, 315 rays, 54 squares, 108 triangles, 279 flaps, and 171 quadrants. \end{corollary}

Here, a ``flap" is a product of a bounded edge and a ray, and a ``quadrant" is a product of two rays.
Note that the Euler characteristic is
$\,117-279+54+108 = 0$.
This is consistent with the fact that the
Berkovich skeleton of $(E \times E)^{an}$ is a torus.

\begin{proof}
Let $X = {\rm trop}(\psi(E)) $
and $\pi : {\rm TGL}(\psi) \rightarrow X \times X$ as
in  Proposition \ref{p:fibers}. We modify
the tropical surface $X\times X$ by attaching the fiber $\pi^{-1}(U,V)$ to each point $(U,V)\in X\times X$.  
These modifications change the polyhedral structure of $X\times X$.  For example, the
torus $\mathbb{S}^1 \times \mathbb{S}^1$ in $X\times X$ 
consists of $36$ squares,
but the modifications
 subdivide each square into two triangles as in Figure~\ref{fig:torus} on the~right.

We give three examples of explicit computations of fibers $\pi^{-1}(U,V)$ but omit the full analysis.    
For convenience, say that a point $U\in X$ {\em prefers} $p_i$ if it lies on the infinite ray associated to $p_i$.  For our first example, suppose $U$ prefers $p_3$ and $V$ prefers $p_6$, and suppose $u,v\in K^*$ are any lifts of $U$ and $V$.  We set
$ \rho = u / p_3 $ and
$\sigma = v / p_6$. These two scalars in $K^*$ satisfy
$\,u^{-1}v^{-1} = p_5 s^{-2} \rho^{-1} \sigma^{-1}\,$ and
\begin{equation}
{\rm val}(1-\rho) = \delta(p_3,u)  > \beta
\quad \hbox{and} \quad
{\rm val}(1-\sigma) = \delta(p_6,v) > \beta.
\label{eq:lengths}
\end{equation}
It is a general fact that
$\,{\rm val}(1-xy) = \min\{{\rm val}(1-x),{\rm val}(1-y)\}\,$ 
 if $x,y \in K^*$ have valuation $0$ and ${\rm val}(1-x)\ne{\rm val}(1-y)$.
 Combining this fact with (\ref{eq:lengths}), we find
    $$ \delta(u^{-1}v^{-1}, p_5) \,=\, {\rm val}(1-s^2\rho\sigma) \,=\,
    {\rm val}(1- s^2) \, =\,  {\rm val}(1-s) \,= \, \beta.$$
We conclude that ${\rm trop}(\psi(u^{-1}v^{-1}))$ prefers $p_5$ and is at lattice distance  $\beta$ from the hexagon.  Thus we do not modify $X\times X$ above  $(U,V)$ since the fiber is a single~point.

\begin{figure}
\centering
\includegraphics[width=2.8in]{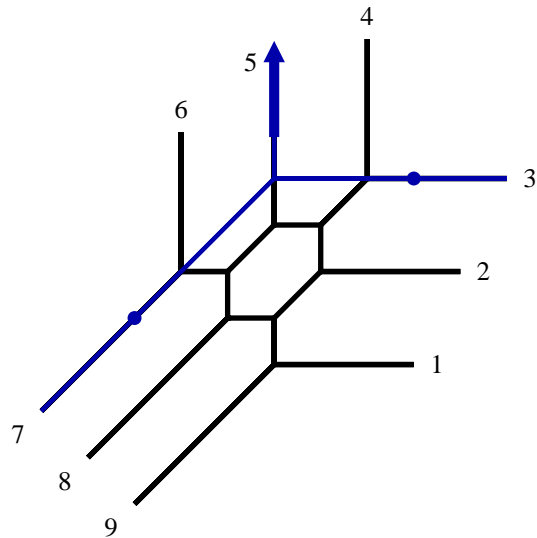}%
\vskip -0.1in
\caption{Constraints on the intersection points of a cubic and a line.}%
\label{fig:line}
\end{figure}

As a second example, suppose $U$ prefers $p_1$ and $V$ prefers $p_3$.
If $u,v\in K^*$ are  lifts of $U$ and $V$ then ${\rm trop}(\psi(u^{-1}v^{-1}))$ prefers $p_2$. Direct computation shows that the minimum of $\delta(u, p_1), \delta(v, p_3), \delta(u^{-1}v^{-1}, p_2)$ is achieved twice,
and this condition characterizes the possibilities for ${\rm trop}(\psi(u^{-1}v^{-1}))$.  For example, if $\delta(u, p_1)= \delta(v, p_3)=d$, then ${\rm trop}(\psi(u^{-1}v^{-1}))$ can be any point that prefers $p_2$ and is distance at least $d$ from the hexagon.  Thus, we modify $X\times X$ at $(U,V)$ by attaching a ray representing the points on the ray of $p_2$ at distance $\geq d$ from the hexagon.

Our third example is similar, but we display it visually.
 Let $U,V$ be the points shown in blue in Figure~\ref{fig:line}, each at distance $ 2\beta$ from the hexagon. 
 Then $\pi^{-1}(U,V)$
  is the thick blue subray that starts at distance $\beta$ from the node
 of the tropical line.

In this way, we modify the surface $X\times X$ at each point $(U,V)$ as prescribed by Proposition~\ref{p:fibers}.  A detailed case analysis  yields the $f$-vector in  Corollary \ref{cor:eleven}.
\end{proof}

We note that
the combinatorics of the tropical group law surface
${\rm TGL}(\psi)$ depends very much on $\psi$.  For example, if $\psi$ is a non-symmetric honeycomb embedding, then the torus in ${\rm TGL}(\psi)$ can contain
quadrilaterals and pentagons,
 as shown in Figure~\ref{fig:torus}.  For this reason, it seems that there is no ``generic" combinatorial description of the surface ${\rm TGL}(\psi)$ as $\psi$ ranges over all honeycomb embeddings.

\medskip

\section*{Acknowledgments}

\noindent
This project grew out of discussions we had with 
Spencer Backman and Matt Baker.
We are grateful for their contributions
and   help with the
analytic theory of elliptic curves.
MC was supported by a Graduate Research Fellowship from
the National Science Foundation.
BS was supported in part by the National Science Foundation 
(DMS-0968882) and the DARPA Deep Learning program (FA8650-10-C-7020).

\bigskip

\bibliographystyle{amsalpha}

\end{document}